\documentclass[a4paper]{amsart}

\usepackage[utf8]{inputenc}
\usepackage[T1]{fontenc}
\usepackage{lmodern, enumerate}
\usepackage{amssymb,amsxtra,amscd,amsmath}
\usepackage{amsthm}
\usepackage{amsfonts}
\usepackage{bbm}
\usepackage[all]{xy}
\usepackage{xcolor}
\usepackage{nicefrac,mathtools}
\usepackage{microtype}
\usepackage[pdftitle={Deformation via coactions},
 pdfauthor={Alcides Buss, Siegfried Echterhoff},
 pdfsubject={Mathematics}
]{hyperref}
\usepackage[lite]{amsrefs}
\usepackage{comment}
\usepackage[normalem]{ulem}

\newcommand*{\MRref}[2]{ \href{http://www.ams.org/mathscinet-getitem?mr=#1}{MR #1}}
\newcommand*{\arxiv}[1]{\href{http://www.arxiv.org/abs/#1}{arXiv: #1}}

\numberwithin{equation}{section}
\theoremstyle{plain}
\newtheorem{theorem}[equation]{Theorem}
\newtheorem{lemma}[equation]{Lemma}
\newtheorem{proposition}[equation]{Proposition}
\newtheorem{corollary}[equation]{Corollary}
\theoremstyle{definition}
\newtheorem{definition}[equation]{Definition}
\theoremstyle{remark}
\newtheorem{remark}[equation]{Remark}
\newtheorem{example}[equation]{Example}

\newtheorem{notation}[equation]{Notation}

\DeclareMathOperator{\Aut}{Aut}
\DeclareMathOperator{\cspn}{\overline{span}}
\DeclareMathOperator{\spn}{{span}}

\DeclareMathOperator{\supp}{\mathrm{supp}}
\DeclareMathOperator{\id}{\mathrm{id}}

\DeclareMathOperator{\Twist}{\mathrm{Twist}}
\DeclareMathOperator{\Br}{\mathrm{Br}}

\newcommand*{\tg}{\mathrm{tg}} 

\newcommand*{\nb}{\nobreakdash}
\newcommand*{\Star}{\(^*\)\nobreakdash-}
\newcommand*{\dd}{\,\mathrm d} 
\newcommand*{\R}{\mathbb R}
\newcommand*{\T}{\mathbb T}
\newcommand*{\Z}{\mathbb Z}

\newcommand*{\C}{\mathbb C}

\newcommand*{\Lb}{\mathbb B}
\renewcommand*{\L}{\mathcal L}
\newcommand*{\K}{\mathbb K}
\renewcommand*{\H}{\mathcal H}

\newcommand*{\cont}{C}
\newcommand*{\contz}{\cont_0}
\newcommand*{\contc}{\cont_c}
\newcommand*{\M}{\mathcal M}

\newcommand{\tig}{{\tilde{g}}}
\newcommand{\tih}{{\tilde{h}}}
\newcommand{\tir}{{\tilde{r}}}

\newcommand*{\Ad}{\textup{Ad}}

\newcommand*{\U}{\mathcal U}
\newcommand*{\E}{\mathcal E}
\newcommand*{\X}{\mathcal X}
\newcommand{\rt}{\mathrm{rt}}

\newcommand*{\congto}{\xrightarrow\sim}

\newcommand*{\braket}[2]{\langle#1\!\mid\!#2\rangle}

\newcommand*{\sbe}{\subseteq} 

\newcommand*{\cstar}{\texorpdfstring{$C^*$\nobreakdash-\hspace{0pt}}{*-}}
\newcommand*{\into}{\hookrightarrow}
\newcommand*{\onto}{\twoheadrightarrow}
\newcommand*{\red}{r}
\renewcommand*{\max}{\mathrm{max}}

\newcommand*{\dual}[1]{\widehat{#1}}
\newcommand*{\dualG}{\widehat{G}}

\newcommand{\om}{\omega}
\newcommand{\Om}{\Omega}

\newcommand{\lk}{\langle}
\newcommand{\rk}{\rangle}
\newcommand{\csp}{\overline{\operatorname{span}}}

\newcommand{\A}{\mathcal A}
\newcommand{\B}{\mathcal B}
\newcommand{\G}{\mathcal G}
\newcommand{\hatG}{\widehat{G}}
\newcommand{\hatdelta}{\widehat{\delta}}

\newcommand{\car}{\curvearrowright}

\newcommand*{\kc}[1]{\Bbbk_c({#1})}
\newcommand*{\kk}[1]{\Bbbk ({#1})}
\newcommand*{\ktwo}[1]{\Bbbk_2 ({#1})}
\newcommand*{\kb}[1]{\Bbbk_{c,b} ({#1})}

\newcommand*{\s}{\mathfrak{s}}

\begin{document}
\title[Deformation of Fell bundles]{Deformation of Fell bundles}

\author{Alcides Buss}
\email{alcides.buss@ufsc.br}
\address{Departamento de Matem\'atica\\
 Universidade Federal de Santa Catarina\\
 88.040-900 Florian\'opolis-SC\\
 Brazil}

\author{Siegfried Echterhoff}
\email{echters@uni-muenster.de}
\address{Mathematisches Institut\\
Universit\"at M\"un\-ster\\
 Einsteinstr.\ 62\\
 48149 M\"unster\\
 Germany}

\begin{abstract}
In this paper we study deformations of \cstar{}algebras that are given as cross-sectional \cstar{}algebras of Fell bundles $\A$ over locally compact groups $G$. Our deformation comes from a direct deformation of the Fell bundles $\A$ via certain parameters, like automorphisms of the Fell bundle, group cocycles,  or central group extensions of $G$ by the circle group $\T$, and then taking cross-sectional algebras of the deformed Fell bundles. We then show that this direct deformation method is equivalent to the deformation via the dual coactions by similar parameters as studied previously in \cites{BNS, BE:deformation}.
\end{abstract}



\thanks{This work was funded by: the Deutsche Forschungsgemeinschaft (DFG, German Research Foundation) Project-ID 427320536 SFB 1442 and under Germany's Excellence Strategy EXC 2044  390685587, Mathematics M\"{u}nster: Dynamics, Geometry, Structure; and CNPq/CAPES/Humboldt - Brazil.}

\maketitle

\begin{center}
    Dedicated to the memory of Iain Raeburn (1949-2023).
\end{center}

\tableofcontents

\section{Introduction}
Inspired by and building on a series of papers by many different authors (e.g., see \cites{Rieffel:Deformation, Rief-K, Kasprzak:Rieffel, Kasprzak1,  BNS, NT}) on deformation of \cstar{}algebras via actions and coactions of locally compact groups with deformation parameters given by $2$-cocycles of these groups, we gave in \cite{BE:deformation} a new description of such deformation by Borel cocycles of locally compact groups which works in a quite general setting. 

Let us recall the basic ideas of this deformation procedure: Suppose that 
$\delta:A\to\M(A\otimes C^*(G))$ is a coaction of the locally compact group $G$ on the \cstar{}algebra $A$. Then the coaction crossed product 
$B:=A\rtimes_{\delta}\widehat{G}$ is equipped with the dual action $\beta:=\widehat{\delta}:G\car B$ and a canonical nondegenerate $\rt-\beta$ equivariant inclusion $\phi:=j_{\contz(G)}:\contz(G)\to\M(B)$, where $\rt\colon G\car \contz(G)$ denotes the action by right translations. The triple $(B, \beta,\phi)$ then provides a set of  data which allows to reconstruct the original cosystem $(A,\delta)$ via an (exotic) version of Landstad duality. Now, if we deform the triple $(B, \beta,\phi)$ by certain parameters ${p}$ to a new Landstad triple $(B_{{p}},\beta_{{p}},\theta_{{p}})$, then applying Landstad duality to  $(B_{{p}},\beta_{{p}},\theta_{{p}})$ provides a deformed cosystem $(A^{{p}}, \delta^{{p}})$. The deformation parameters $p$ can be certain group actions $\alpha: G\car A$ by \Star{}automorphisms, or Borel $2$-cocycles $\om$ of $G$, or, related to the latter, central extensions $\sigma=(\T\into G_\sigma\onto G)$ of $G$ by the circle group $\T$. 
If we restrict ourselves to normal (or reduced) coactions, i.e., those, for which the composition 
$$\delta^\lambda:=(\id_A\otimes \lambda)\circ \delta:A\to \M(A\otimes C_r^*(G))$$
is faithful, a similar procedure has been given before by Bhowmick, Neshveyev, and Sangha in \cite{BNS}, using cocycles $\om\in Z^2(G,\T)$ as deformation parameters. In that paper they also observed that for $G$ discrete, deformation by $\om$ can also be obtained via deformation of the underlying Fell bundle as studied by Yamashita in \cite{Yamashita}.
Recall from \cites{Quigg:Discrete, Ng} that for every normal coaction $\delta:A\to A\otimes C^*(G)$ of a discrete group $G$, there exists a unique Fell bundle $p:\A\to G$ such that the given coaction identifies with the dual coaction $\delta:C_r^*(\A)\to C_r^*(\A)\otimes C^*(G)$ 
 on the reduced cross-sectional algebra $A:=C_r^*(\mathcal A)$.
Then, given any cocycle $\om \in Z^2(G,\T)$, the deformed Fell bundle 
$\mathcal A_\om$ is constructed by introducing a twisted  multiplication $*_\om$ on the original Fell bundle  $\A$ via the formula
$$a_g*_\om a_h=\om(g,h)a_g a_h,$$
for $g,h\in G$, $a_g\in A_g$, $a_h\in A_h$, where $a_ga_h$ denotes the 
product in $\A$. Then the deformed coaction $(A^\om,\delta^\om)$ can be described as the dual coaction 
$(C_r^*(\A_\om), \delta_\om)$ on the deformed Fell bundle $\A_\om$. 

If $G$ is locally compact and $\om\in Z^2(G,\T)$ is a {\em continuous} cocycle, then a similar deformation of a Fell bundle $\mathcal A$ over $G$ makes perfect sense, and it is not too difficult to see that passing to reduced cross-sectional algebras in this setting will coincide with the 
deformation by $\om$ of $A=C_r^*(\mathcal A)$ via the dual coaction $\delta: A\to \M(A\otimes C^*(G))$
in the sense of \cite{BNS} or \cite{BE:deformation}.

The main objective of this paper 
is to describe a direct deformation procedure on the level of Fell bundles for groups of automorphisms $\alpha\colon G\car \A$, following an approach to deformation given by Abadie and Exel in \cite{Abadie-Exel:Deformation}, and for twists $\sigma=(\T\into G_\sigma\onto G)$ for $G$, 
which then covers also the case of general Borel $2$-cocycles on $G$, and to compare the outcome with the deformation via  dual coactions 
as studied for  general duality crossed-product functors in \cite{BE:deformation} (like, for instance, the dual coaction on the maximal cross-sectional algebra $C_{\max}^*(\A)$ of the Fell bundle $\A$). This will provide a convenient direct deformation process in this situation. 
We then use continuity and $K$-theory results obtained in \cite{BE:deformation} in the setting of deformation via coactions
for the cross-sectional algebras of our deformed Fell bundles. In the special case of Fell bundles over  discrete amenable groups $G$, we recover a beautiful result of Raeburn in \cite{Raeburn:Deformations}, where he constructed a continuous bundle of deformed cross-sectional algebras 
$C^*(\A_\om)$ over the second Borel-comomology group $H^2(G,\T)$.

\subsection*{Acknowledgement} Both authors wish to express their deep gratitude for the profound insights they have gained from Iain Raeburn’s mathematical legacy. The second author had the pleasure of many stimulating joint projects with Iain from which he not only learned some deep mathematics, but also the desire to write the papers in a style which should be understandable for a large readership. Unfortunately, he never succeeded to reach Iain's mastership in this respect  (and in others). Iain's death is a huge loss for the operator algebras community!

Most of this work was written while the first author was visiting the University of Münster, and he is deeply grateful to the entire group -- especially the second author -- for their warm hospitality!

\section{Actions, coactions and their (exotic) crossed products}\label{sec-prel}

For terminology and notation concerning (co)actions, their (exotic) crossed products, and duality -- particularly Landstad duality for coactions in terms of generalized fixed-point algebras -- we refer the reader to our previous paper \cite{BE:deformation}.

Let us just recall some basic notation and terminology. 
Throughout the paper, $G$ will usually denote a locally compact group, with a fixed Haar measure.
Continuous actions of $G$ on a \cstar{}algebra $B$ will be usually shortly written as $\beta\colon G\car B$.

In what follows below, by an {\em (exotic) crossed product} $B\rtimes_{\beta,\mu}G$ for an action $\beta:G\car B$ we understand  
a \cstar{}completion of $C_c(G,B)$ by a norm which satisfies $\|f\|_r\leq\|f\|_\mu\leq \|f\|_{\max}$ for all $f\in C_c(G,B)$, where $\|\cdot\|_r$ (resp.~$\|\cdot\|_{\max}$) denote the reduced (resp.~ maximal) crossed-product norms on $C_c(G,B)$. A crossed product $B\rtimes_{\beta,\mu}G$ is called a {\em duality crossed product} if the dual coaction $\widehat{\beta}$ on $B\rtimes_{\beta,\max}G$ factors through a coaction $\widehat{\beta}_\mu$ on $B\rtimes_{\beta,\mu}G$.
A {\em crossed-product functor} is a functor $(B,\beta)\mapsto B\rtimes_{\beta,\mu}G$ from the category of $G$-algebras to the category of 
\cstar{}algebras that sends actions $\beta:G\car B$ to crossed products $B\rtimes_{\beta,\mu}G$ such that for any $G$\nb{}-equivariant \Star{}homomorphism $\Phi:(B,\beta)\to (B',\beta')$ the associated \Star{}homomorphism $\Phi\rtimes_\mu G:B\rtimes_{\beta,\mu}G\to B'\rtimes_{\beta',\mu}G$
extends $\Phi\rtimes_{alg}G: C_c(G,B)\to C_c(G,B'); f\mapsto \Phi\circ f$. 
If all $\rtimes_\mu$-crossed products are duality crossed products, then $\rtimes_\mu$ is called a 
{\em duality crossed-product functor}, which are the crossed-product functors we consider in this work. 
However, in \S 8 below, we need to restrict our attention to crossed-product functors that are also functorial with respect to 
($G$-equivariant) correspondences. These are called \emph{correspondence crossed-product functors} and include the maximal and reduced crossed products.
It is shown in \cite{BEW}*{Theorem~5.6} that all correspondence crossed-product functors are duality functors.

A coaction of $G$ on a \cstar{}algebra $A$  will  usually be denoted by the symbol
$\delta\colon A\to \M(A\otimes C^*(G))$ and its crossed product will be denoted by $A\rtimes_\delta\dualG$. Recall that $A\rtimes_\delta\dualG$ can be 
realized as 
$$\overline{\spn}\big((\id\otimes\lambda)\circ\delta(A)(1\otimes \M(C_0(G))\big)\subseteq \M(A\otimes \K(L^2(G)),$$
where $M:C_0(G)\to \mathbb{B}(L^2(G))$ is the representation by multiplication operators. 
We  often write 
$$j_A:=(\id\otimes\lambda)\circ \delta:A\to \M(A\rtimes_\delta\dualG)\;\text{and}\; j_{C_0(G)}:=1\otimes M:C_0(G)\to \M(A\rtimes_\delta\dualG)$$
for the canonical morphisms from $A$ and $C_0(G)$ into $\M(A\rtimes_\delta\dualG)$. The dual action $\widehat{\delta}:G\car A\rtimes_{\delta}\dualG$ 
is then determined  by the equation 
$$\widehat\delta_g\big(j_A(a)j_{C_0(G)}(f)\big)=j_A(a)j_{C_0(G)}(\rt_g(f)),$$
where $\rt:G\car C_0(G)$ denotes the action by right translations. 

It has been observed by Nilsen in \cite{Nilsen:Duality}*{Corollary~2.6} that for every coaction $\delta:A\to\M(A\otimes C^*(G))$ there exists a canonical 
surjective \Star{}homomorphism
$$\Psi_A: A\rtimes_\delta\dualG\rtimes_{\widehat\delta,\max}G\onto A\otimes \K(L^2(G))$$
given as the integrated form of the covariant representation $(j_A\rtimes j_{C_0(G)}, 1\otimes\rho)$.
The coaction  $\delta$ is called {\em maximal} if $\Psi_A$ is an isomorphism, and it is called {\em normal} if it factors through an isomorphism $ A\rtimes_\delta\dualG\rtimes_{\widehat\delta,r}G\congto A\otimes \K(L^2(G))$. In general, it factors through an isomorphism
\begin{equation}\label{eq:Katayama-mu-duality} A\rtimes_\delta\dualG\rtimes_{\widehat\delta,\mu}G\congto A\otimes \K(L^2(G))
\end{equation}
for some (possibly exotic) duality crossed product $\rtimes_\mu$. We then say that $(A,\delta)$ is a $\mu$-coaction to indicate that it satisfies Katayama duality for the $\mu$-crossed product.

The triple $(A\rtimes_\delta\dualG,\widehat{\delta}, j_{C_0(G)})$ is what we call a \emph{weak} $G\rtimes G$-algebra. More generally, a \emph{weak} $G\rtimes G$-algebra is a triple $(B, \beta, \phi)$ where:
\begin{itemize}
    \item $B$ is a \cstar-algebra,
    \item $\beta:G\curvearrowright B$ is an action of $G$ on $B$,
    \item $\phi:C_0(G)\to \M(B)$ is a nondegenerate, $\rt-\beta$-equivariant $*$-homomorphism.
\end{itemize}
As a variant of the classical Landstad duality for reduced coactions \cite{Quigg:Landstad}, it is shown in \cite{Buss-Echterhoff:Exotic_GFPA} that for any given \emph{duality crossed product} $B\rtimes_{\beta,\mu}G$, there exists a unique (up to isomorphism) $\mu$-coaction $(A_\mu, \delta_\mu)$ of $G$ such that  
\[
(A_\mu\rtimes_{\delta_\mu}\dual G, \widehat{\delta}_\mu, j_{C_0(G)})\cong (B, \beta, \phi).
\]
This provides the main tool for \emph{deformation by coactions}, as introduced in \cite{BE:deformation}, which serves as the foundation for most of the results presented in this paper. The construction of $(A_\mu, \delta_\mu)$ is carried out using the theory of \emph{generalized fixed-point algebras} and depends on the choice of the crossed product $B\rtimes_{\beta,\mu}G$. Moreover, if we start with a duality crossed-product functor $\rtimes_\mu$ on the category of $G$-\cstar{}algebras, it is shown in \cite{BE:deformation}*{Proposition~2.9} (see also \cite{Buss-Echterhoff:Exotic_GFPA}*{Lemma~7.1}) that the assignment  
\[
(B,\beta,\phi) \mapsto (A_\mu,\delta_\mu)
\]  
is functorial. More precisely, given weak $G\rtimes G$-algebras $(B,\beta,\phi)$ and $(B',\beta',\phi')$, if $\Psi:B\to B'$ is a $\beta-\beta'$-equivariant $*$-homomorphism satisfying $\Psi\circ \phi=\phi'$, then $\Psi$ induces a canonical $\delta_\mu-\delta'_\mu$-equivariant $*$-homomorphism  
\[
\psi:A_\mu\to A'_\mu
\]  
between the corresponding $\mu$-fixed-point algebras. In particular, we obtain the following result.  

\begin{proposition}\label{prop-iso-fixed}
Suppose that $(B,\beta,\phi)$ and $(B',\beta',\phi')$ are isomorphic weak $G\rtimes G$-algebras, and let $\rtimes_\mu$ be a duality crossed-product functor.  
Then the corresponding $\mu$-coactions $(A_\mu, \delta_\mu)$ and $(A'_\mu,\delta'_\mu)$ via Landstad duality are also isomorphic.
\end{proposition}

The above proposition will serve as our main tool for comparing different deformation procedures in this work.

\section{Fell bundles}

Recall from \cites{Doran-Fell:Representations,Doran-Fell:Representations_2}  that a  Fell bundle $\mathcal A$ over the locally compact group $G$  is a collection of Banach spaces $\{A_s: s\in G\}$ together  with 
a set of pairings (called multiplications) 
 $A_s\times A_t\to A_{st}: (a_s, a_t)\mapsto a_sa_t$  and involutions $A_s\to A_{s^{-1}}; a_s\mapsto a_s^*$ which are compatible with the linear structures in the usual sense known from \cstar{}algebras, including the condition 
 $\|a_sa_s^*\|=\|a_s\|^2$ for all $a_s\in A_s$. In particular, it follows that the fibre $A_e$ over the unit $e\in G$ is a \cstar{}algebra.
 If $G$ is not discrete, the topological structure of $\mathcal A$ is determined by the set $C_c(\mathcal A)$ of {\em compactly supported continuous sections $a: G\to \A$, $s\mapsto a_s\in A_s$}. 
 We refer to \cites{Doran-Fell:Representations,Doran-Fell:Representations_2,Exel:Book} for more details on this structure. The space $C_c(\A)$ 
 becomes a \Star{}algebra when equipped with convolution and involution given by
 $$ (a*b)_t=\int_G a_sb_{s^{-1}t}\,\dd s\quad\text{and}\quad a^*_t=\Delta(t^{-1})a_{t^{-1}}^*
\quad a,b\in C_c(\A).$$
A representation of $\A$ into $\M(D)$ for some \cstar{}algebra $D$ is a mapping $\pi:\A\to \M(D)$ that preserves multiplication and involution, and such that for every $d\in D$ and every  $a\in C_c(\A)$ the map $G\to D; s\mapsto \pi(a_s)d$ is continuous.
Such representation is {\em nondegenerate} if $\pi(A_e)D=D$, i.e., the restriction of $\pi$ to the unit fibre $A_e$ of $\A$ is a nondegenerate 
representation of the \cstar{}algebra $A_e$. Every representation $\pi:\A\to \M(D)$ integrates to give a \Star{}representation 
$$\pi:C_c(\A)\to \M(D);\quad \pi(a)d:=\int_G \pi(a_s)d\, \dd s \quad\forall a\in C_c(\A), d\in D$$
which is nondegenerate (in the sense that $\csp \pi(C_c(\A))D=D$)  if and only if $\pi$ is nondegenerate. 
The maximal (or universal) cross-sectional \cstar{}algebra $C^*(\A)$ is then defined as the completion of $C_c(\mathcal A)$ 
with respect to the \cstar{}norm
$$\|a\|_{\max}:=\sup_{\pi}\|\pi(a)\|$$
where $\pi$ runs through all possible (nondegenerate) representations of $\A$. Passing to the integrated form and extension to the completion $C^*(\A)$ of $C_c(\A)$ then gives a one-to-one correspondence between nondegenerate representations of $\A$ and nondegenerate \Star{}representations of $C^*(\A)$.  We refer to \cite{Buss-Echterhoff:Maximality}*{Proposition 2.1} for a list of alternative characterizations of the representations of $\A$. 

On the other extreme we have the {\em reduced cross-sectional \cstar{}algebra} $C_r^*(\A)$ which can be described as the image of $C^*(\A)$ under the 
{\em (left) regular representation} $\Lambda_\A:C^*(\A)\to \Lb_{A_e}(L^2(\A))$, with $\Lambda_\A(a)\xi=a*\xi$ for $a\in C_c(\A)\subseteq C^*(\A)$ and $\xi\in \contc(\A)\sbe L^2(\A)$, where $L^2(\A)$ denotes the Hilbert $A_e$-module obtained as a completion of $C_c(\A)$ by the $A_e$-valued inner product
$$\lk \xi, \eta\rk_{A_e}=\int_G \xi(s)^* \eta(s)\,\dd s.$$  

There exists a \emph{dual coaction} 
$$\delta_{\A}: C^*(\A)\to \M(C^*(\A)\otimes C^*(G))$$
given by the integrated form of the representation 
$s\mapsto a_s\otimes u_s$, where, as above,  $u:G\to U\M(C^*(G))$ denotes the universal representation of $G$ and each $a_s\in A_s$ acts as a multiplier of $C^*(\A)$ by $(a_s\cdot f)(t):=a_sf(s^{-1}t)$. It has been shown in \cite{Buss-Echterhoff:Maximality}*{Theorem 3.1} that $\delta_\A$ is a maximal coaction, that is, it satisfies
Katayama duality~\eqref{eq:Katayama-mu-duality} with respect to the maximal crossed product $\rtimes_{\max}$. 

Now, similar to \cite{Buss-Echterhoff:Maximality}*{Definition 4.1}, given any duality crossed-product functor $\rtimes_\mu$, there is a unique quotient $C_\mu^*(\A)$ of 
$C^*(\A)$ such that $\delta_\A$ factors through a coaction 
$$\delta_\mu:C_\mu^*(\A)\to \M(C_\mu^*(\A)\otimes C^*(G))$$
and such that $(C_\mu^*(\A),\delta_\mu)$ satisfies  Katayama duality for the  $\mu$-crossed product as in~\eqref{eq:Katayama-mu-duality}. In particular, the dual coaction $(C_r^*(\A), \delta_r)$ on the reduced cross-sectional algebra $C_r^*(\A)$ corresponds to the reduced crossed-product functor $\rtimes_r$ in this way. 
It is the {\em normalization} of $(C^*(\A), \delta_\A)$ in the sense of Quigg (e.g., see \cite{Buss-Echterhoff:Maximality}). 
We also have $(C^*_\max(\A),\delta_{\max})=(C^*(\A),\delta_\A)$. 
In particular, $C_\mu^*(\A)$ is an ``exotic'' completion of the cross-sectional \Star{}algebra $C_c(\A)$ that sits between $C^*(\A)$ and $C^*_r(\A)$ in the sense that the identity map on $C_c(\A)$ extends to surjective \Star{}homomorphisms
$$C^*(\A)\onto C^*_\mu(\A)\onto C^*_r(\A).$$
Combining this with \cite{Buss-Echterhoff:Exotic_GFPA}*{Theorem 4.3} 
we get 

\begin{proposition}\label{prop-mu-cross-sectional-algebra}
The coaction  $(C_\mu^*(\A),\delta_\mu)$  coincides with the coaction provided by applying $\mu$-Landstad duality  to the weak $G\rtimes G$-algebra
$(C^*(\A)\rtimes_{\delta_\A}\dualG, \dual{\delta_\A}, j_{C_0(G)})$.
\end{proposition}

\subsection{The C*-algebra of kernels of a Fell bundle}\label{sec:kernels-algebra}

We need to recall the realization of $C_r^*(\A)\rtimes_{\delta_r} \hatG$ (and hence also  of $C_\mu^*(\A)\rtimes_{\delta_\mu}\hatG$) as a completion
of a  certain \Star{}algebra of  kernel functions $k:G\times G\to \A$ due to Abadie (see \cite{Abadie:Enveloping}*{\S 5}).
For this let 
\begin{equation*}\label{eq-kc}
\kc{\A}:=\{k:G\times G\to \A:\mbox{$k$ is cont. with compact supports and } k(s,t)\in A_{st^{-1}}\}.
\end{equation*}
In other words, $\kc\A=C_c(\nu^*(\A))$, the space of compactly supported continuous sections of the pullback 
$\nu^*(\A)$ of $\A$ along $\nu\colon G\times G\to G,\, (s,t)\mapsto st^{-1}$. This is a \Star{}algebra with convolution and involution given by
\begin{equation}\label{eq-convinvkernel}
k*l(s,t)=\int_G k(s,r)l(r,t)\dd r\quad\mbox{and}\quad k^*(s,t)=k(t,s)^*
\end{equation}
for all $k,l\in \kc{\A}$ and $s,t\in G$.  By a continuous \Star{}representation
of $\kc{\A}$ we understand a \Star{}homomorphism $\pi: \kc{\A}\to D$ for some \cstar{}algebra $D$ such that
$$\|\pi(k)\|\leq \|k\|_2:=\left(\int_{G\times G} \|k(s,t)\|^2\,\dd(s,t)\right)^{1/2}$$
for all $k\in \kc{\A}$. Let $\|k\|_u:=\sup_{\pi}\|\pi(k)\|$, where $\pi$ runs through all continuous \Star{}representations
of $\kc{\A}$. Then the completion $\kk{\A}$ of $\kc{\A}$ by this norm is a \cstar{}algebra; this was introduced by Abadie in \cite{Abadie:Enveloping}. As noticed there, $\kk\A$ can also be viewed as the enveloping \cstar{}algebra of the Banach \Star{}algebra obtained as the completion of $\kc\A$ with respect to $\|\cdot\|_2$. We write $\ktwo{\A}$ for this $L^2$-completion and note that it coincides with the $L^2$-completion of the space $\kb{\A}$ of all bounded compactly supported measurable functions $k\colon G\times G\to \A$, which satisfy the condition $k(s,t)\in A_{st^{-1}}$. 
The following result is due to Abadie (\cite{Abadie:Enveloping}):

\begin{proposition}\label{prop-kernels}
For every duality crossed-product functor $\rtimes_\mu$, there
 is a canonical isomorphism $C_\mu^*(\A)\rtimes_{\delta_\mu}\hatG\cong \kk{\A}$  which sends
a typical element of the form $j_{C^*_\mu(\A)}(a) j_{\contz(G)}(f)$ with $a\in \contc(\A)$ and $f\in \contc(G)$ to the kernel function
\begin{equation}\label{eq-kernel}
k_{a,f}(s,t) = a(st^{-1})f(t)\Delta(t^{-1}).
\end{equation}
Viewing $\kc{\A}$ as a dense subalgebra of $C_\mu^*(\A)\rtimes_{\delta_\mu}\hatG$ using this isomorphism,
the dual action $\hatdelta_\mu$ of $G$  is given  on $\kc{\A}$ by the formula
\begin{equation}\label{eq-hatdelta}
\hatdelta_{\mu,r}(k)(s,t)=\Delta(r)k(sr,tr).
\end{equation}
Moreover, we have
\begin{equation}\label{eq-jGk}
(j_{\contz(G)}(f)k)(s,t)=f(s)k(s,t)\quad\text{and}\quad (kj_{\contz(G)}(f))(s,t)=k(s,t)f(t)
\end{equation}
for all $k\in \kc{\A}$ and $f\in C_0(G)$.
\end{proposition}
\begin{proof}
Abadie only considers \emph{reduced coactions}, that is, injective coactions of the reduced group \cstar{}algebra $C^*_\red(G)$.
It is well known that such coactions correspond bijectively (in a naturally functorial way) to normal coactions of $C^*(G)$ in our sense (e.g., see \cite{EKQR}*{Appendix A.9}).
Moreover, this bijective correspondence preserves crossed products and their representation theory.
That said, what Abadie proves in \cite{Abadie:Enveloping} is that there is an isomorphism $C^*_\red(\A)\rtimes_{\delta_r}\dualG\cong\kk{\A}$ which is given as in the statement (see the proof of Proposition~8.1 in \cite{Abadie:Enveloping}). But this implies the general version as in the statement for every exotic \cstar{}norm associated to a duality crossed-product functor $\rtimes_\mu$ by \cite{BE:deformation}*{Theorem~2.4}.
\end{proof}

The following result describes the representations of $\kk\A$.

\begin{proposition}\label{prop:covariance-Fell-bundle-dual-coaction}
    Let $D$ be a \cstar{}algebra. A pair $(\pi,\kappa)\colon (C^*(\A),\contz(G))\to \M(D)$ of nondegenerate representations is covariant for a dual coaction $\delta_{\A}$ of $G$ on the cross-sectional \cstar{}algebra of a Fell bundle $\A$, and hence extends to a nondegenerate representation $\pi\rtimes\kappa\colon C^*(\A)\rtimes_{\delta_\A}\dualG\cong\kk\A\to\M(D)$ if and only if
    \begin{equation}\label{eq:covariance-Fell-bundle}
        \pi(a)\kappa(_sf)=\kappa(f)\pi(a)\quad\forall a\in A_s, f\in \contz(G),
    \end{equation}
    where $_s f(t):=f(st)$ denotes left translation, and we use the same notation $\pi\colon \A\to \M(D)$ for the disintegrated form of $\pi\colon C^*(\A)\to\M(D)$.
\end{proposition}
\begin{proof}
    By definition, $(\pi,\kappa)$ is covariant if and only if
    \begin{equation}\label{eq:covariance-coaction}
        (\pi\otimes\id)(\delta_\A(a))=w_\kappa(\pi(a)\otimes 1)w_\kappa^*,
    \end{equation}
    for all $a\in C^*(\A)$, where $w_\kappa:=(\kappa\otimes \id)(w_G)$ and $w_G\in \M(\contz(G)\otimes C^*(G))$ is given by the universal representation $s\mapsto u_s$. Since $\pi$ is nondegenerate, it extends to the multiplier algebra $\M(C^*(\A))$. Using the inclusion $A_s\into \M(C^*(\A))$ given by  $(a_s\cdot \xi)(t)=a_s\xi(s^{-1}t)$, for $a_s\in A_s$, $\xi\in C_c(\A)$,  
    and  the formula $\delta_\A(a_s)=a_s\otimes u_s$, the covariance conditon for $(\pi,\kappa)$ is equivalent to
    $$(\pi(a_s)\otimes u_s)w_\kappa=w_\kappa(\pi(a_s)\otimes 1)\quad \forall a_s\in A_s,\, s\in G.$$
    Now we remark that if we view the Fourier algebra $A(G)\subseteq C_0(G)$ as a subalgebra of $B(G)\cong C^*(G)'$, the 
    set of continuous linear functionals on $C^*(G)$, then
    $\kappa(f)=(\id\otimes f)(w_\kappa)$ for all $f\in A(G)$.
    Taking slices with $f\in A(G)$ we see that the covariance condition is equivalent to
    $$\pi(a_s)(\id\otimes f)((1\otimes u_s)w_\kappa)=\kappa(f)\pi(a_s)\quad\forall a_s\in A_s, s\in G, f\in A(G).$$
    But 
    \begin{align*}
    (\id\otimes f)((1\otimes u_s)w_\kappa)&=(\kappa\otimes f)((1\otimes u_s)w_G)\\
    &=(\kappa\otimes \id)(\id\otimes f)((1\otimes u_s)w_G)=\kappa(f\cdot u_s),    
    \end{align*}
    where $(f\cdot u_s)(x):=f(u_sx)$ for all $x\in C^*(G)$. Viewed as a function on $G$ this gives $f\cdot u_s={_s}f$.
    Therefore the covariance condition is equivalent to
    $$\pi(a_s)\kappa(_s f)=\kappa(f)\pi(a_s)\quad\forall a_s\in A_s, s\in G, f\in A(G).$$
    Since $A(G)$ is dense in $\contz(G)$, this is equivalent to~\eqref{eq:covariance-Fell-bundle}.
\end{proof}

It was shown by Abadie \cite{Abadie:Enveloping} that there exists a well-defined \Star{}representation  $\Lambda_{\kk{\A}}:\kk{\A}\to  \Lb_{A_e}(L^2(\A))$,
called the {\em regular representation} of $\kk\A$, given by
$$ \big(\Lambda_{\kk{\A}}(k)\xi\big)(t)=\int_G k(s,t)\xi(t)\,\dd t, \quad\text{for $k\in \kc{\A},\xi\in C_c(\A)$}.$$
It is stated in \cite{Abadie:Enveloping}*{Theorem~5.1(2)} that $\Lambda_{\kk{\A}}$ is always faithful. But,
unfortunately, this is {\em not true} in general (although it is true in several important cases, e.g., if $\A$ is saturated).\footnote{One can find more information about this in the arxiv version of the paper \cite{Abadie:Enveloping} available at \url{https://arxiv.org/pdf/math/0007109.pdf}, where the wrong statements have been fixed.} As a counterexample let $\A$ be the Fell bundle over $\Z_2$ with fibres $\A_0=\C$ and $\A_1=0$. Then $\kk\A=\C\oplus\C$, while $L^2(\A)=\C$. 

We now provide a different representation of $\kk\A$ that will be faithful in general.
For this we shall use the fact that $\kk\A$ is isomorphic to the crossed product $C^*(\A)\rtimes_{\delta_\A}\dualG$.
Let
$$\delta_\A^\lambda:=(\Lambda_\A\otimes\lambda)\circ\delta_{\A}\colon C^*(\A)\to \M(C_r^*(\A)\otimes C^*_r(G))$$
denote the reduction of the dual coaction $\delta_\A$ of $G$ on $C^*(\A)$ (which factors through a genuine reduced coaction of $C_r^*(G)$ on $C_r^*(\A)$). 
The algebra $\M(C_r^*(\A)\otimes C^*_r(G))$ is clearly represented faithfully on $L^2(\A)\otimes L^2(G)$ via $\Lambda_\A\otimes \lambda$.
Now $L^2(\A)\otimes L^2(G)$ can be identified with  $L^2(\A\times G)$ if  $\A\times G$ denotes the Fell bundle over 
$G\times G$ given by the pullback of $\A$ via the projection  $G\times G\to G;(g,h)\mapsto g$, so in what follows we will regard $\Lambda_\A\otimes \lambda$  as a representation of $\M(C_r^*(\A)\otimes C_r^*(G))$ into $\Lb_{A_e}(L^2(\A\times G))$.
Then a short computation on the fibres $A_s\subseteq \A$ shows (see \cite{ExelNg:ApproximationProperty}*{Lemma 2.9 and Proposition 2.10}) that
$$\delta_\A^\lambda(a)=W_\A(\Lambda(a)\otimes 1)W_\A^*\quad\mbox{for all }a\in C^*(\A),$$
where $W_\A\in \Lb_{A_e}(L^2(\A\times G))$ is the unitary operator defined by
$$W_\A\zeta(s,t):=\zeta(s,s^{-1}t).$$
Now recall that the crossed product $C^*(\A)\rtimes_{\delta_\A}\dualG\cong C_r^*(\A)\rtimes_{\delta_r}\dualG$ can be realized as
$$C^*(\A)\rtimes_{\delta_\A}\dualG=\cspn\{\delta_\A^\lambda(a)(1\otimes M_f): a\in C_c(\A), f\in \contc(G)\}\sbe\Lb_{A_e}(L^2(\A\times G)).$$
A simple computation shows that
$$W_\A^*(1\otimes M_f)W_\A=\tilde M_{f},\quad \mbox{where }\tilde M_f\zeta(s,t):=f(st)\zeta(s,t),$$
for all $f\in \contz(G)$ and $\zeta\in \contc(\A\times G)\sbe L^2(\A\times G)$. It follows that
$$\delta_\A^\lambda(a)(1\otimes M_f)=W_\A(\Lambda(a)\otimes 1)\tilde M_f W_\A^*.$$
Hence, conjugation by the unitary $W_\A$ yields an isomorphism
$$C^*(\A)\rtimes_{\delta_\A}\dualG\cong \cspn\{(\Lambda(a)\otimes 1)\tilde M_f: a\in \contc(\A), f\in \contz(G)\}\sbe \Lb_{A_e}(L^2(\A\times G)).$$
On the other hand, the operators $(\Lambda(a)\otimes 1)\tilde M_f$ can be computed as
$$\big((\Lambda(a)\otimes 1)\tilde M_f\zeta\big)(s,t)=\int_G a(r)\big(\tilde M_f\zeta\big)(r^{-1}s,t)\dd r
=\int_G a(r)f(r^{-1}st)\zeta(r^{-1}s,t)\dd r$$
$$=\int_G a(sr^{-1})f(rt)\Delta(r)^{-1}\zeta(r,t)\dd r=\int_G \beta_t(k_{a,f})(s,r)\zeta(r,t)\dd r,$$
where $k_{a,f}(s,r):=a(sr^{-1})f(r)\Delta(r^{-1})$ as in~\eqref{eq-kernel}, and $\beta$ denotes the dual action of $G$ on $\kk{\A}$, as in Proposition~\ref{prop-kernels}. Using the isomorphism $\kk{\A}\cong C^*(\A)\rtimes_{\delta_\A}\dualG$ provided by this proposition, we obtain 

\begin{proposition}\label{prop:formula-T-rep}
Given a Fell bundle $\A$ over $G$, there exists a faithful representation
$$T\colon \kk{\A}\to \Lb_{A_e}(L^2(\A\times G))$$
given on $\kc\A$ (resp. $\kb\A$) by the formula
$$T_k\zeta(s,t):=\int_G \beta_t(k)(s,r)\zeta(r,t)\dd r=\int_G k(st,rt)\Delta(t)\zeta(r,t)\dd r.$$
\end{proposition}

\begin{corollary}\label{cor:structural-hom-ext}
The structural homomorphism $\phi\colon \contz(G)\to \M(\kk\A)$ extends to a $G$-equivariant faithful unital \Star{}homomorphism
$$\bar\phi\colon L^\infty(G)\to \M(\kk\A)$$
given by the formulas 
$(\bar\phi(f)k)(s,t)=f(s)k(s,t)$, $(k\bar\phi(f))(s,t)=k(s,t)f(t)$
for $f\in L^\infty(G), k\in \kb\A$.
\end{corollary}
\begin{proof} Consider the $*$-representation $\Phi:L^\infty(G)\to \Lb_{A_e}(L^2(\A\times G))$ 
given  by $\big(\Phi(f)\xi\big)(s,t)=f(s)\xi(s,t)$ for 
$f\in L^\infty(G)$ and $\xi\in L^2(\A\times G)$. We then compute
$$T_{\bar\phi(f)k}=\Phi(f)T_k\quad\text{and}\quad T_{k\bar\phi(f)}=T_k\Phi(f)$$
for all $f\in L^\infty(G)$ and $k\in \kb\A$. This gives the result.
\end{proof}

\section{Deformation}

In \cite{BE:deformation} we  described a general process to deform \cstar{}algebras via coactions as follows: if $\delta\colon A\to \M(A\otimes C^*(G))$ is a coaction of the locally compact group $G$ on the \cstar{}algebra $A$, we consider the associated weak $G\rtimes G$-algebra 
$$(B,\beta,\phi)=(A\rtimes_\delta\dualG,\dual\delta, j_{C_0(G)}).$$
We assume that $(A,\delta)$ satisfies Katayama duality for a given duality crossed-product functor $\rtimes_\mu$ for actions of $G$, as in~\eqref{eq:Katayama-mu-duality}. Given a suitable deformation parameter -- such as a (Borel) $2$-cocycle on $G$ or an action of $G$ on $A$ that commutes with $\delta$ -- we deform the weak $G\rtimes G$-algebra $(B,\beta,\phi)$ into another weak $G\rtimes G$-algebra $(B',\beta',\phi')$ accordingly.
Applying Landstad duality to $(B',\beta',\phi')$, as developed in \cite{Buss-Echterhoff:Exotic_GFPA}, with respect to the crossed-product functor $\rtimes_\mu$, we obtain a deformed coaction $(A',\delta')$ such that
\[
(B',\beta', \phi') = (A'\rtimes_{\delta'}\dualG, \dual{\delta'}, j'_{C_0(G)}),
\]
where $(A',\delta')$ also satisfies Katayama duality with respect to $\rtimes_\mu$.

\subsection{Abadie-Exel deformation}\label{subsec-EA}

Following ideas of Abadie and Exel from \cite{Abadie-Exel:Deformation}, we introduced in \cite{BE:deformation}*{Section~3.2} a deformation process for a given coaction $\delta\colon A\to \M(A\otimes C^*(G))$ that uses a (strongly) continuous action $\alpha:G\car A$ commuting with $\delta$
in the sense that
$$\delta(\alpha_s(a))=(\alpha_s\otimes \id_G)(\delta(a))\quad\mbox{ for all }a\in A\mbox{ and }s\in G.$$
Given such action $\alpha$, the equation
 \begin{equation}\label{eq:definition-tilde-alpha}
\tilde\alpha_s\big(j_A(a)j_{\contz(G)}(f)\big):=j_A(\alpha_s(a))j_{\contz(G)}(f)\quad a\in A, f\in C_0(G)
\end{equation}
 determines an action $\tilde\alpha:G\car B=A\rtimes_{\delta}\hatG$
 that commutes with the dual action $\dual\delta$ (meaning $\tilde\alpha_s\circ\dual\delta_t=\dual\delta_t\circ\tilde\alpha_s$ for all $s,t\in G$) and fixes $\phi=j_{C_0(G)}$. Starting then from the weak $G\rtimes G$-algebra $(B,\beta,\phi)=(A\rtimes_\delta\dualG, \dual\delta, j_{C_0(G)})$, we 
 obtain a new weak $G\rtimes G$-algebra $(B, \tilde\alpha\cdot\beta,\phi)$ with $(\tilde\alpha\cdot\beta)_s:=\tilde\alpha_s\circ \beta_s$ for all $s\in G$. 
Conversely,  \cite{BE:deformation}*{Lemma 3.3} shows that every action $\eta:G\car B$ that commutes with $\beta=\dual\delta$ and fixes $\phi$ is equal to $\tilde\alpha$ for some action $\alpha:G\car A$ commuting with $\delta$.

\begin{definition}\label{def-AE-deformation}
Let $(A,\delta)$ be a $\mu$-coaction with respect to a duality crossed-product functor $\rtimes_\mu$ and let 
$\alpha:G\car A$ be an action which commutes with $\delta$ as above. Let 
$(B,\beta, \phi)=(A\rtimes_\delta\widehat{G}, \widehat{\delta}, j_{\contz(G)})$ and let
 $(A^\alpha, \delta^\alpha)$ be the 
 coaction obtained from the weak $G\rtimes G$-algebra $(B, \tilde\alpha\cdot\beta,\phi)$
 via Landstad duality with respect to $\rtimes_\mu$.
Then we call $(A^\alpha, \delta^\alpha)$ the 
{\em Abadie-Exel deformation} of $(A,\delta)$ with respect to $\alpha$.

\end{definition}

We call this {\em Abadie-Exel deformation} because it covers the deformation of
cross-sectional algebras of Fell bundles as studied by Abadie and Exel in \cite{Abadie-Exel:Deformation}, as we are  going to show in the next section.

\subsection{Abadie-Exel deformation via Fell bundles}\label{AE-Fell}

Let $\A=(A_s)_{s\in G}$ be a Fell bundle over the locally compact group $G$. 
Abadie and Exel considered in \cite{Abadie-Exel:Deformation} continuous actions $\alpha:G\car \A$ of $G$ by automorphisms of the Fell bundle $\A$.
This means that for each $s,t\in G$ we get a Banach-space isomorphism
$$\alpha_s^t:A_t\to A_t; a_t\mapsto \alpha^t_s(a_t)$$
such that
\begin{enumerate}
\item $\alpha^t_s(a_t)\alpha^r_s(a_r)=\alpha^{tr}_s(a_ta_r)$ and $\alpha^{t^{-1}}_s(a_t^*)=\alpha_s^t(a_t)^*$ for all $a_t\in A_t, a_r\in A_r$;
\item for each fixed section $a\in \contc(\A)$ the section $\alpha_s(a)$ defined by
$$\big(\alpha_s(a)\big)_t:=\alpha_s^t(a_t)$$ is continuous, and
\item  the function $s\mapsto \alpha_s(a)$ is continuous
with respect to $\|\cdot\|_\max$, the universal norm on $\contc(\A)$ (for this it suffices that the function $s\mapsto\alpha_s(a)$
is continuous in the $L^1$-norm on $\contc(\A)$).
\end{enumerate}
These properties imply that $\alpha$ induces  an action (also called $\alpha$) of $G$ on $C^*(\A)$ which  commutes with 
$\delta_\A$ since
$\delta_\A\circ \alpha_s$ is the integrated form of the representation $a_t\mapsto \alpha^t_s(a_t)\otimes u_t$ which  coincides with 
the representation $(\alpha_s\otimes \id_G)\circ \delta_\A$. 

Let now $\rtimes_\mu$ be a duality crossed-product functor for $G$.
Since 
$$(C_\mu^*(\A)\rtimes_{\delta_\mu}\widehat{G}, \dual{\delta_\mu}, j_{\contz(G)})=(C^*(\A)\rtimes_{\delta_\A}\widehat{G}, \widehat{\delta_\A}, j_{\contz(G)})=:(B,\beta,\phi)$$
it follows from \cite{BE:deformation}*{Lemma~3.3} (using both directions of that lemma) 
that the action $\alpha$ factors through an action (still called $\alpha$) 
on $C_\mu^*(\A)$ which commutes with $\delta_\mu$. 
 We then get the deformed cosystem  $(A^\alpha, \delta^\alpha)$
 as in Definition \ref{def-AE-deformation} starting from 
 $(A,\delta):=(C_\mu^*(\A), \delta_\mu)$.

On the other hand, using the action $\alpha$ on $\A$ we can now define a new multiplication and involution
on the Fell bundle $\A$  by
$$a_s*_\alpha a_t=a_s\alpha_s(a_t)\quad\mbox{and}\quad a_s^{*_\alpha}=\alpha_{s^{-1}}(a_s^*),$$
where from now on we shall simply write $\alpha_s(a_t)$ instead of $\alpha_s^t(a_t)$.
We  write $\A_\alpha$ for the new Fell bundle obtained in this way. 
We write
$\delta^\alpha_\mu$ for the dual coaction on $C_\mu^*(\A_\alpha)$. Similar to our notations for the  Fell bundle $\A$, we write
$(C^*(\A_\alpha), \delta_{\A_\alpha})$ for the maximal coaction $(C_{\max}^*(\A_\alpha), \delta^\alpha_{\max})$. 

\begin{theorem}\label{thm-AE-Fell}
There is a canonical isomorphism  between the cosystems 
$(C_\mu^*(\A_\alpha),\delta_\mu^\alpha)$ and  $(A^\alpha,\delta^\alpha)$ as constructed above  from the action $\alpha:G\car \A$.
\end{theorem}

\begin{remark}\label{rem-AE}
In \cite{Abadie-Exel:Deformation} Abadie and Exel define the deformed Fell bundles $\A_\alpha$ 
as above for arbitrary locally compact groups, but only discuss their cross-sectional \cstar{}algebras for discrete abelian groups. 
They then define the deformation of $(A,\delta)$ via  $\alpha:G\car\A$ as the cosystem $(C^*(\A_\alpha), \delta_{\A_\alpha})$.
The above construction extends this to arbitrary locally compact groups and  other possible completions of $C_c(\A)$ and $C_c(\A_\alpha)$, respectively.
\end{remark}

\begin{proof}[Proof of Theorem \ref{thm-AE-Fell}]
Recall that we constructed 
$(A^\alpha, \delta^\alpha)$ via Landstad duality for the weak $G\rtimes G$-algebra 
$(B, \tilde\alpha\cdot\beta, \phi)$ with respect to $\rtimes_\mu$, with 
$(B,\beta,\phi):=(C^*(\A)\rtimes_{\delta_\A}\widehat{G}, \widehat{\delta_\A}, j_{\contz(G)})$. 
On the other hand, by Proposition \ref{prop-mu-cross-sectional-algebra}, the cosystem $(C_\mu^*(\A_\alpha),\delta_\mu^\alpha)$ can be constructed similarly via the
weak $G\rtimes G$ algebra 
$$(B^\alpha, \beta^\alpha, \phi^\alpha):=(C^*(\A_\alpha)\rtimes_{\delta_{\A_\alpha}}\dualG,\dual{\delta_{\A_\alpha}}, j^\alpha_{C_0(G)}).$$
So the result will follow immediately from Proposition \ref{prop-iso-fixed} if we can show that  
\begin{equation}\label{iso-to-show}
(C^*(\A)\rtimes_{\delta_\A}\widehat{G}, \tilde\alpha\cdot \dual{\delta_\A}, j_{\contz(G)})\cong (C^*(\A_\alpha)\rtimes_{\delta_{\A_\alpha}}\dualG,\dual{\delta_{\A_\alpha}}, j^\alpha_{C_0(G)})
\end{equation}
as weak $G\rtimes G$-algebras.
Using Proposition \ref{prop-kernels}, for this it suffices to show that there is an isomorphism
between $\kc{\A}$ and $\kc{\A_\alpha}$ which is isometric for $\|\cdot\|_2$ and
 which intertwines the relevant actions of $G$ and $C_0(G)$, respectively.
For this we define
$$\Phi_\alpha:\kc{\A_\alpha}\to \kc{\A},\quad \Phi_\alpha(k)(s,t)=\alpha_{s^{-1}}(k(s,t))$$
for all $k\in \kc{\A_\alpha}$. This is clearly a linear bijection which is isometric for the norm $\|\cdot\|_2$.
But it is also a \Star{}homomorphism: for $k,l\in \kc{\A_\alpha}$ we compute
\begin{align*}
\Phi_\alpha(k*l)(s,t)&=\alpha_{s^{-1}}(k*l(s,t))=\alpha_{s^{-1}}\left(\int_G k(s,r)\alpha_{sr^{-1}}(l(r,t))\dd r\right)\\
&=\int_G\alpha_{s^{-1}}(k(s,r))\alpha_{r^{-1}}(l(r,t))\dd r\\
&=\Phi_\alpha(k)*\Phi_\alpha(l)(s,t),
\end{align*}
for all $s,t\in G$. We also have
\begin{align*}
\Phi_\alpha(k^*)(s,t)&=\alpha_{s^{-1}}(k^*(s,t))=\alpha_{s^{-1}}(k(t,s)^{*_\alpha})\\
&=\alpha_{t^{-1}}(k(t,s))^*=\Phi_\alpha(k(t,s))^*=\Phi_\alpha(k)^*(s,t).
\end{align*}
This proves the existence of the isomorphism $\Phi_\alpha$ as stated in the theorem.
It is trivial to check that $\Phi_\alpha$ intertwines the  $C_0(G)$-action $j_{\contz(G)}^\alpha$ with $j_{\contz(G)}$  and
that it is is $\widehat{\delta_{\A_\alpha}}-\tilde\alpha \cdot \hatdelta_\A$ equivariant. Hence it preserves the
weak $G\rtimes G$-structures on both algebras. Since Landstad duality with respect to $\rtimes_\mu$ is functorial by
\cite{Buss-Echterhoff:Exotic_GFPA}, this finishes the proof.
\end{proof}

\begin{example}
    Let $A$ be a \cstar{}algebra and consider the trivial Fell bundle $\A=A\times G$ over $G$ with fibres 
    $A_t=A\times\{t\}\cong A$ for all $t\in G$ and with multiplication and involution  given by 
    $(a,s)(b,t)=(ab,st)$ and $(a,s)^*=(a^*, s^{-1})$.
    Any $G$-action $\alpha\colon G\car A$ extends to a $G$-action on $\A$ by $\alpha^t_s(a, t):=(\alpha_s(a),t)$. The induced $G$-action on $C^*(\A)\cong A\otimes_\max C^*(G)$ is just $\alpha\otimes \id$, and the dual $G$-coaction on $C^*(\A)$ corresponds to $\id\otimes\delta_G$. The $\alpha$-deformed Fell bundle $\A_\alpha$ is easily seen to be isomorphic to the semi-direct product Fell bundle $\A_\alpha=A\times_\alpha G$ associated to the action $\alpha\colon G\car A$. From our general Theorem~\ref{thm-AE-Fell} we get that the $\alpha$-deformation of the coaction $(C^*(\A),\delta_\A)\cong (A\otimes_\max C^*(G),\id\otimes\delta_G)$ is isomorphic to $(C^*(\A_\alpha),\delta_{\A_\alpha})\cong (A\rtimes_{\alpha,\max}G,\dual\alpha)$. Similarly, taking reduced norms, we get that $(C^*_r(\A),\delta_{\A}^r)\cong (A\otimes C^*_r(G),\id\otimes\delta_G^r)$ is deformed into $(A\rtimes_{\alpha,r}G,\dual\alpha_r)$, and a similar result holds for exotic crossed products (where, in general, $A\rtimes_{\id,\mu}G$ has no obvious description as `standard' tensor product).

    More generally, we can start with an arbitrary action $\beta\colon G\car A$ on a fixed \cstar{}algebra $A$, and then any other $G$-action $\alpha\colon G\car A$ commuting with $\beta$ gives an action on the Fell bundle $\A_\beta=A\times_\beta G$ with $\alpha$-deformed Fell bundle $(\A_\beta)_\alpha=\A_{\beta\cdot\alpha}=A\times_{\beta\cdot\alpha}G$, where $\beta\cdot\alpha\colon G\car A$ is defined by $(\beta\cdot\alpha)_t:=\beta_t\circ\alpha_t=\alpha_t\circ\beta_t$. In this situation, given a duality crossed-product functor $\rtimes_\mu$ for $G$, we deform the dual $G$-coaction on $C^*_\mu(\A_\beta)=A\rtimes_{\beta,\mu}G$ into the dual $G$-coaction on $C^*_\mu(\A_{\beta\cdot\alpha})=A\rtimes_{\beta\cdot\alpha,\mu}G$.    
\end{example}

\begin{example}
       As a particular example, consider the trivial Fell bundle $\cont(\T)\times \Z$ and the action 
       $\alpha_\theta$ of $\Z$ on $\cont(\T)$ by rotations for a fixed angle $\theta\in \R$. 
       In this case we deform $\cont(\T)\otimes C^*(\Z)\cong \cont(\T^2)$ with the dual $\T$-action (that is, a $\Z$-coaction) 
    into the rotation algebra $\cont(\T^2_\theta):=A\rtimes_{\alpha_\theta}\Z$ with its dual $\T$-action. 
    It is interesting to notice that $\T$-actions (i.e. $\Z$-coactions) cannot be deformed using $2$-cocycles (i.e. twists) as in the next section because the group $\Z$ carries no nontrivial $2$-cocycles. But we do get $\cont(\T^2_\theta)$ as a cocycle deformation of $\cont(\T^2)\cong C^*(\Z^2)$ considering it endowed with its canonical (dual) $\T^2$-action and deforming it with respect to the $2$-cocycle on $\Z^2$ given by $\omega_\theta((n,m),(k,l))=e^{2\pi i\theta mk}$. 
\end{example}

We see from the above family of examples that deformation of Fell bundles in the sense of Abadie-Exel is quite general. In particular, we should not expect as many permanence results to hold in this setting, as we have for deformation by $2$-cocycles as studied in  \cite{BE:deformation} or in Section \ref{sec-Boreltwist} below.

 Of course, one can also combine Abadie-Exel deformations with cocycle deformations to obtain
 more examples. The numerous examples as discussed in the paper 
 \cite{Abadie-Exel:Deformation} by Abadie and Exel show that this theory has many interesting applications.
 
\section{Deformation by twists}\label{sec-Boreltwist}
Here we recall the deformation of coactions via twists, following the approach introduced in \cite{BE:deformation}. Let $G$ be a locally compact group. We first recall that a \emph{twist} over $G$ is a central extension
\begin{equation}\label{eq:twist-def}
    \sigma=(\T\stackrel{\iota}{\into} G_\sigma\stackrel{q}{\onto} G)
\end{equation}
of $G$ by the circle group $\T$. In what follows, we often omit the embedding $\iota$ in our notations by simply identifying $\T$ as a (central) normal subgroup of $G_\sigma$ and  $q$ with the quotient map $G_\sigma\onto G_\sigma/\T\cong G$.

Given a Borel section $\s\colon G\to G_\sigma$ for $q$ which satisfies $\s(e)=1$, we obtain a (normalized) Borel cocycle $\om=\om_\sigma\in Z^2(G,\T)$ by $\omega(g,h):=\s(g)\s(h)\s(gh)^{-1}$. The cohomology class $[\om_\sigma]\in H^2(G,\T)$ does not depend on the choice of the Borel section $\mathfrak{s}$
and we obtain an isomorphism $[\sigma]\leftrightarrow  [\om_\sigma]$ between the group $\Twist(G)$ of isomorphism classes $[\sigma]$ 
of twists, equipped with the Baer multiplication of extensions, and the group $H^2(G,\T)$. We refer to  \cite{BE:deformation}*{Section 4} for a detailed discussion.
There it is also shown that $\Twist(G)\cong H^2(G,\T)$ are  isomorphic to the Brauer group $\Br(G)$ of Morita equivalence classes of actions $\alpha:G\car \K(\H)$ on algebras of compact operators on some Hilbert space $\H$. It sends a class $[\om]\in H^2(G,\T)$ to the 
Morita equivalence class $[\alpha]$ of the action $\alpha=\Ad\rho^{\bar\om}:G\car \K(L^2(G))$, 
where $\rho^{\bar\om}:G\car \U(L^2(G))$ is the {$\bar\om$-}twisted right regular 
representation of $G$ given by the formula $\big(\rho^{\bar\om}_s\xi\big)(t)=\Delta(s)^{\frac12}\bar\om(t,s)\xi(ts)$, and $\bar\om$ is the complex conjugate (i.e., inverse) of $\om$. In this paper we shall mostly use the picture of twists given by central extensions~\eqref{eq:twist-def}, and sometimes also use $2$-cocycles.

Given a twist $G_\sigma$ over $G$, we  can consider the Green-twisted action of $(G_\sigma ,\T)$ on $\C$  given  by 
the pair  $(\id_\C, \iota^\sigma)$ in which $\iota^\sigma:\T\to \T=U(\C)$ denotes the identity map. Then,
if $\beta:G\car B$ is any action of $G$, we can {\em twist} $\beta$ with $(\id_{\C},\iota^\sigma)$ by taking the
diagonal twisted action 
\begin{equation}\label{eq-twistedbeta}
(\beta,\iota^\sigma):=\beta\otimes (\id_{\C},\iota^\sigma):(G_\sigma, \T)\car B\otimes \C\cong B,
\end{equation}
where we identify $\beta$ with its inflation $\beta\circ q\colon G_\sigma\car B$.
Note that  in this setting we have $\iota^\sigma_z b=zb$ for all $z\in \T, b\in B$.

We now fix a duality crossed-product functor $\rtimes_\mu$  for $G$ and a $\mu$-coaction  $(A,\delta)$.
Consider, as before, the corresponding weak $G\rtimes G$-algebra 
$$(B,\beta,\phi)=(A\rtimes_{\delta}\dualG, \dual{\delta}, j_{C_0(G)}).$$
We want to use this twisted action to construct a deformed weak $G\rtimes G$-algebra 
$(B_\sigma,\beta_\sigma, \phi_\sigma)$ as follows:
we first observe that the 
twist $\sigma=(\T\into G_\sigma\stackrel q\onto G)$ determines a
complex line bundle $\L_\sigma$ over $G$ given by the quotient space
\begin{equation}\label{eq-Lom}
\L_\sigma=G_\sigma\times_\T \C:=\big(G_\sigma\times \C\big)/\T,
\end{equation}
with respect to the action
$\T\car G_\sigma\times \C; \;z\cdot\big(x, u\big)=\big(\bar zx, zu\big)$.
The $C_0$-sections of this bundle then naturally identify with the space of functions
\begin{equation}\label{eq-line-bundle-sections}
C_0(G_\sigma, \iota):=\{f\in C_0(G_\sigma): f(zx)=\bar{z} f(x)\mbox{ for all } x\in G_\sigma, z\in \T\}.
\end{equation}

We observe in \cite{BE:deformation}*{Remark~5.6} that $C_0(G_\sigma, \iota)$ becomes a $C_0(G)-C_0(G)$  imprimitivity Hilbert bimodule 
and the right translation action $\rt^\sigma:G_\sigma \car C_0(G_\sigma, \iota)$ implements a (Morita) equivalence between the right translation action $\rt: G\car C_0(G)$ on the left and 
 the twisted right translation action $(\rt,\iota^\sigma):(G_\sigma,\T)\car \contz(G)$  on the right. 
We then observe in \cite{BE:deformation}*{Remark~5.11} that for the weak $G\rtimes G$-algebra $(B,\beta, \phi)$, the internal tensor product
$$\L(G_\sigma,B):= C_0(G_\sigma,\iota)\otimes_{C_0(G)}B$$
is a full Hilbert $B$-module and the diagonal action $\rt^\sigma\otimes\beta$  of $G_\sigma$ on
$\L(G_\sigma,B)$ is compatible with the  twisted action $(\beta,\iota^\sigma):(G_\sigma,\T)\car B$ as introduced in~\eqref{eq-twistedbeta} above.
The adjoint action $\beta_\sigma:=\Ad(\rt^\sigma\otimes\beta)$ then factors through an ordinary action of $G$ on $B_\sigma:=\K(\L(G_\sigma,B))$ and $(\L(G_\sigma,B), \rt^\sigma\otimes\beta)$  becomes 
a $\beta_\sigma-(\beta, \iota^\sigma)$  equivariant Morita equivalence.

Together with the left action of $\phi_\sigma: C_0(G)\to\Lb(\L(G_\sigma,B))\cong \M(B_\sigma)$, which is induced from the left action of 
$C_0(G)$ on $C_0(G_\sigma,\iota)$, the triple $(B_\sigma, \beta_\sigma,  \phi_\sigma)$ 
becomes a  weak $G\rtimes G$-algebra. The following is \cite{BE:deformation}*{Definition 5.8}.

 \begin{definition}\label{def-cosystem}
Let $(A,\delta)$ be a $\mu$-coaction for some duality crossed-product functor $\rtimes_\mu$ and let $(B_\sigma,\beta_\sigma, \phi_\sigma)$ be the weak 
  $G\rtimes G$-algebra constructed from the twist $\sigma=(\T\into G_\sigma\onto G)$ and from $(B,\beta, \phi):=(A\rtimes_\delta \widehat{G}, \widehat{\delta}, j_{\contz(G)})$ 
  as above. 
 We then  define the {\em $\sigma$-deformation} of $(A,\delta)$ as  the cosystem 
 $(A^\sigma_\mu, \delta^\sigma_\mu)$  associated to $(B_\sigma,\beta_\sigma,\phi_\sigma)$ and $\rtimes_\mu$ via Landstad duality with respect to $\rtimes_\mu$.
  \end{definition} 

In other words, $(A^\sigma_\mu,\delta^\sigma_\mu)$ is the unique $\mu$-coaction for which there exists an isomorphism of weak $G\rtimes G$-algebras
$$(A^\sigma_\mu\rtimes_{\delta^\sigma_\mu}\dualG,\dual{\delta^\sigma_\mu},\phi_{A^\sigma_\mu})\cong (B_\sigma,\beta_\sigma,\phi_\sigma).$$

\begin{remark}
    If we start with an arbitrary coaction $(A,\delta)$ and any possibly unrelated duality crossed-product functor $\rtimes_\mu$, 
    it follows from $\mu$-Landstad duality applied to $(B,\beta,\phi)=(A\rtimes_\delta\dualG,\dual\delta,j_{C_0(G)})$  that there is a unique (up to isomorphism)
    $\mu$-coaction $(A_\mu,\delta_\mu)$ such that $(B,\beta,\phi)\cong (A_\mu\rtimes_{\delta_\mu}\dualG,\dual{\delta_\mu},j_{C_0(G)})$.
    The  $\sigma$-deformed coaction $(A^\sigma_\mu,\delta^\sigma_\mu)$ associated to $(B_\sigma,\beta_\sigma,\phi_\sigma)$ is then the  $\sigma$-deformation 
    of  $(A_\mu,\delta_\mu)$.
\end{remark}

If the twist $\sigma=(\T\into G_\sigma\stackrel{q}\onto G)$ splits via a continuous section $\s\colon G\to G_\sigma$, then $\sigma$ is represented in $H^2(G,\T)$ by the class of
the continuous $2$-cocycle $\omega(g,h)=\partial \mathfrak{s}(g,h)=\s(g)\s(h)\s(gh)^{-1}$.  In this case the deformed weak $G\rtimes G$-algebra $(B_\sigma,\beta_\sigma,\phi_\sigma)$ is isomorphic to
the triple $(B, \beta^\om,\phi)$ in which  $\beta^\omega:G\car B$ is given by the formula
\begin{equation}\label{eq:twisted-action-omega}
\beta^\omega:G\to \Aut(B); \quad\beta^\omega(s)=\Ad U_\om(s)\circ \beta(s),
\end{equation}
where $U_\om:G\to U\M(B)$ is defined by $U_\om(s)= \phi(u_\om(s))$, with $u_\om(s)\in C_b(G,\T)=U\M(C_0(G))$ given by
\begin{equation}\label{eq-uomega}
u_\om(s)(r)=\overline{\om(r,s)}.
\end{equation}
We refer to  \cite{BE:deformation}*{Secton 3.3} for further details.

Indeed, for a general twist $\sigma=(\T\into G_\sigma\stackrel{q}\to G)$, by \cite{FG}*{Theorem~1}, we can always choose a {\em Borel} section $\s\colon G\to G_\sigma$ for $q$ with corresponding Borel cocycle $\om=\partial \mathfrak{s}$ and
it may happen that the action~\eqref{eq:twisted-action-omega} makes sense even if $\omega$ is not continuous. This is the case, for example, if the structural homomorphism $\phi\colon \contz(G)\to \M(B)$ extends to a $G$-equivariant unital homomorphism $\bar\phi\colon L^\infty(G)\to \M(B)$; this is the content of \cite{BE:deformation}*{Proposition~5.13}. 
We are going to apply this in the next section to the kernel algebras $\kk{\A}$ using Corollary~\ref{cor:structural-hom-ext} in order to deform dual coactions on cross-sectional \cstar{}algebras of Fell bundles $\A$ over $G$.

\section{Deformation of Fell bundles by cocycles}\label{subsec-Fellcocycle}

Let $\rtimes_\mu$ be a duality crossed-product functor for a locally compact group $G$. 
Given a Fell bundle $\A$ over $G$, we consider its \cstar{}cross-sectional algebra $C_\mu^*(\A)$, which comes equipped with the dual $\mu$-coaction $\delta:C_\mu^*(\A)\to \M(C_\mu^*(\A)\otimes C^*(G))$ 
given by the integrated form of the map  $a_s\mapsto a_s\otimes u_s$ for $a_s\in A_s$.

In this section we want to show that the deformation of $C_\mu^*(\A)$ via this coaction and a twist $\sigma=(\T\into G_\sigma\onto G)$ can be described directly on the level of the Fell bundle $\A$ itself -- similarly to what happened in  Section~\ref{subsec-EA} for the case of Abadie-Exel deformation. This result will generalize \cite{BE:deformation}*{Proposition~5.16}, which in turn generalizes \cite{BNS}*{Proposition 4.3}, where a similar result has been shown for  dual coactions of reduced crossed products, and 
also results from \cite{Yamashita} for the case of reduced cross-sectional \cstar{}algebras of Fell bundles over discrete groups.

So in what follows let $\A=(A_g)_{g\in G}$ be a Fell bundle over $G$ and $A=C_\mu^*(\A)$ with respect to some duality crossed-product functor $\rtimes_\mu$ for $G$.  
We construct the deformed Fell bundle $\A_\sigma$ via the twist $\sigma=(\T\stackrel{\iota}{\into} G_\sigma\stackrel{q}{\onto} G)$ as follows:
consider the pullback Fell bundle over $G_\sigma$, that can be realized as 
$$q^*(\A)=\{(a,\tilde g)\in \A\times G_\sigma: p(a)=q(\tilde g)\},$$
where $p\colon \A\to G$ denotes the bundle projection. The circle group $\T$ acts `diagonally' on $q^*(\A)$ via
$$z\cdot (a,\tilde g):=(\bar z a,z\tilde g),\quad z\in\T,\, a\in \A,\, \tilde g\in G_\sigma.$$
\begin{definition}
    We define the \emph{twisted Fell bundle} $\A_\sigma$ as the quotient Fell bundle 
    $$\A_\sigma:=q^*(\A)/\T$$
    with respect to the diagonal $\T$-action defined above. More concretely, this consists of equivalence classes $[a,\tilde g]$, with the bundle projection 
    $$p^\sigma\colon \A_\sigma\to G; \; p^\sigma[a,\tilde g]:=p(a)=q(\tilde g).$$
    The algebraic operations on $\A_\sigma$ are also induced from those in $q^*(\A)$, that is,
    $$\lambda\cdot [a,\tilde g]+[a',\tilde g]:=[\lambda a+a',\tilde g],$$
    $$[a,\tilde g]\cdot [b,\tilde h]:=[ab,\tilde g\tilde h], \quad [a,\tilde g]^*:=[a^*,\tilde g^{-1}].$$
    for all $\lambda\in \C$, $a,a'\in A_g$, $b\in A_h$ $\tilde g,\tilde h\in G_\sigma$, $g=q(\tilde g)$ and $h=q(\tilde h)$.
\end{definition}

It is a well-known result that $\A_\sigma$ is a Fell bundle over $G$ with respect to  the above operations, see for instance \cite{Kaliszewski-Muhly-Quigg-Williams:Fell_bundles_and_imprimitivity_theoremsI}*{Corollary~A.12}.

\begin{remark}\label{rem:line-Fell-bundle}
    The fibers of $\A_\sigma$ are $A_{\sigma,g}\cong A_g$ as Banach spaces for each $g\in G$. Indeed, choosing $\tilde g\in G_\sigma$ with $q(\tilde g)=g$, the map $a\mapsto [a,\tilde g]$ provides an isometric linear isomorphism $A_g\congto A_{\sigma,g}$. However, $\A$ and $\A_\sigma$ are not isomorphic as topological bundles in general, as shown by the following special case:
    
    If we apply the above construction to the trivial Fell bundle $\A=\C\times G$, we get the twisted Fell line bundle $\L_\sigma$ as  in~\eqref{eq-Lom}, and this is topologically trivial if and only if the twist $\T\into G_\sigma\stackrel{q}\onto G$ admits a continuous section $\s\colon G\to G_\sigma$. Moreover, every Fell line bundle over $G$ is isomorphic to one of this form, see \cite{Doran-Fell:Representations_2}*{VIII.16.2}. 
    
    One can show that, in general, the Fell bundle $\A_\sigma$ is isomorphic to the Fell bundle tensor product $\A\otimes_G \L_\sigma$. 
   Recall that if $\A$ and $\B$ are two Fell bundles over $G$, the (minimal) tensor product $\A\otimes_G\B$ is the Fell bundle over $G$ 
   with fibres $A_g\otimes B_g$ for $g\in  G$, where $\otimes$ here denotes the (external) minimal tensor product of Hilbert modules and multiplication and involution are 
  given on elementary tensors in the obvious way.  Since we do not need this description, we omit further details. Instead, we refer to \cite{Abadie:Tensor} where general tensor products of Fell bundles over groups are discussed, and \cite{KLQ:tensor}*{Section~3} where a similar construction of a (balanced) tensor product of Fell bundles over locally compact groups is used for maximal tensor products.

Applying the construction to the (semidirect product) Fell bundle $\A=A\times_\alpha G$ associated to a (continuous) action $\alpha$ of $G$ on a \cstar{}algebra $A$, the twisted Fell bundle $\A_\sigma$ corresponds to a semidirect product Fell bundle $A\times_{(\alpha,\iota^\sigma)}G$ for the Green-twisted action $(\alpha, \iota^\sigma)$ of $(G_\sigma,\T)$ similar to~\eqref{eq-twistedbeta}.\footnote{It follows from  \cite{Doran-Fell:Representations_2}*{VIII.6}, see also \cite{Buss-Meyer:Crossed}*{Example 3.9}, that every Green twisted action of a pair $(H,N)$, with $N\trianglelefteq H$ a closed normal subgroup,   corresponds to a Fell bundle over $G=H/N$.}
\end{remark}

We are now going to describe the space of sections of the twisted Fell bundle $\A_\sigma$. We start with

\begin{lemma}\label{lem:eq-omsections}
There is a bijection between the sections $\xi:G\to \A_\sigma$ and the set of sections 
\begin{equation}\label{eq-section1}
S(G_\sigma, \A,\iota):=\{\xi:G_\sigma\to \A: \text{$\xi(\tig)\in A_{q(\tig)}$ and $\xi(\tig z)=\bar{z}\xi(\tig)\; \forall \tig\in G_\sigma, z\in \T$}\}.
\end{equation}
for $q^*\A$ which assigns to each section $\xi\in S(G_\sigma, \A,\iota)$ the section 
\begin{equation}\label{eq-section2}
\tilde\xi:G\to \A_{\sigma}; \; \tilde\xi(g)=[\xi(\tig), \tig]\quad \text{with $\tig\in q^{-1}(g)$}.
\end{equation}
Moreover, $\xi$ is continuous (resp.~measurable, resp.~compactly supported)  if and only if $\tilde\xi$ is continuous (resp.~measurable, resp.~compactly supported).
\end{lemma}

\begin{proof}  We first check that the section $\tilde\xi:G\to \A_\sigma$ as in~\eqref{eq-section2} is well defined. Indeed, if $\tig, g'\in q^{-1}(g)\in G_{\sigma}$, then there exists a unique $z\in \T$ such that $g'=\tig z$, so that
$$[\xi(g'), g']=[\tilde\xi(\tig z), \tig z]=[\bar{z}\xi(\tig), z\tig]=[\xi(\tig), \tig].$$
Thus $\tilde\xi$ sends the element $g\in G$ to a well-defined element in the fibre $A_{\sigma,g}$ of $A_\sigma$. 
Conversely, if $\tilde\xi:G\to\A_\sigma$ is any section, then for each $g\in G$ there exist  $a_g\in A_g$ and $g'\in q^{-1}(g)$ 
such that $\tilde\xi(g)=[a_g, g']\in A_{\sigma,g}$. But then,  for any $\tig\in q^{-1}(g)$, there is a unique $z\in \T$ 
such that $g'=\tig z$ and then we get 
$$[a_g, g']=[a_g, \tig z]=[za_g, \tig].$$ 
It follows that $\xi(\tig):=za_g$ is the unique element 
in $A_g$ such that $\tilde\xi(g)=[\xi(\tig), \tig]$. We therefore obtain a well-defined element $\xi\in S(G_\sigma,\A,\iota)$
which satisfies \eqref{eq-section2}. 

We now show that continuity of $\tilde\xi$ implies continuity of $\xi$. For this let $(\tig_i)$ be a net in 
$G_\sigma$ which converges to $\tig\in G_\sigma$, and let us write $g_i=q(\tig_i)$ and $g=\tig$.
Then $g_i\to g$ and continuity  of $\tilde\xi$ implies that 
$\tilde\xi(g_i)=[\xi(\tig_i), \tig_i]\to [\xi(\tig), \tig]=\tilde\xi(g)$. 
Since the quotient map $q^*(\A)\onto q^*(\A)/\T=\A_\sigma$ is open we may assume, after passing to a subnet if necessary,  that there is a net  $(z_i)$ in $\T$ such that
$z_i(\xi(\tig_i),\tig_i)= (\xi(z_i\tig_i), z_i\tig_i) \to (\xi(\tig),\tig)$ in $q^*(\A)$. 
But since $\T$ is compact we may assume without loss of generality that
 $z_i\to z$ for some $z\in \T$. But then it follows that $\tig_i\to \bar{z}\tig$, and since $\tig_i\to \tig$ by assumption, we get $z=1$
and hence $\lim_i \xi(\tig_i)=\lim_i\xi(z_i\tig_i)=\xi(\tig)$, which proves continuity of $\xi$. 

Conversely, if $\xi$ is continuous and $g_i\to g$ in $G$, using openness of $q:G_\sigma\onto G$ we may pass to a subnet
to find a net $(\tig_i)$ in $G_\sigma$ and $\tig\in G_\sigma$ such that $\tig_i\to \tig$ and $q(\tig_i)=g_i, q(\tig)=g$.
It then follows that $\tilde\xi(g_i)=[\xi(\tig_i), \tig_i]\to [\xi(\tig),\tig]=\tilde\xi(g)$ and therefore $\tilde\xi$ is continuous as well.

To see that $\tilde\xi$ is measurable if and only if $\xi$ is measurable we recall from \cite{Doran-Fell:Representations}*{Proposition 15.4}
that $\xi$ is measurable if for every compact subset $K$ of $G_\sigma$  there exists a sequence $(\xi_n)$ of continuous sections
such that $\xi_n(\tig)\to \xi(\tig)$ for almost all $\tig \in K$ (and similarly for $\tilde\xi$). But $\xi_n(\tig)\to \xi(\tig)$ for almost all $\tig\in K$ if and only if $\tilde\xi_n(g)\to \tilde\xi(g)$ for 
almost all $g\in q(K)$, and the result follows.

Finally,  it follows from~\eqref{eq-section2} that $\supp\xi=q^{-1}(\supp\tilde\xi)$. Thus  $\xi$ has compact support if and only if $\tilde\xi$ has
compact support. 
\end{proof}

\begin{lemma}\label{lem-crosssectional}
The convolution algebra $C_c(\A_\sigma)$ is isomorphic to the algebra
$$C_c(G_\sigma, \A, \iota):=\{a\in S(G_\sigma,\A,\iota): \text{$a$ is continuous with compact supports}\}$$
with convolution and involution given by 
$$(a*b)_{\tig}=\int_{G} a_{\tilde h} b_{\tilde{h}^{-1}\tig}\, \dd h\quad\text{and}\quad a^*_{\tig}=\Delta(\tig) (a_{\tig^{-1}})^*,$$
where $\tilde{h}\in G_\sigma$ with $h=q(\tilde h)$. 
Moreover, every section $a\in C_c(G_\sigma,\A,\iota)$ can be written as a finite linear combination of functions 
of the form $\tig\mapsto f(\tig)\tilde{a}_g$, $g=q(\tig)$, with $f\in C_c(G_\sigma, \iota)$ and $\tilde{a}\in C_c(\A)$.
\end{lemma}
\begin{proof} It follows from Lemma \ref{lem:eq-omsections} that 
$$C_c(G_\sigma,\A,\iota)\to C_c(\A_\sigma); a\mapsto \Big(g\mapsto [a_{\tig},\tig]\Big)\;\text{if $g=q(\tig)$}$$
is a linear bijection. We need to check that it preserves convolution and involution, respectively.
We first observe that by the conditions on the elements in $S(G_\sigma, \A, \iota)$ the integrand in the 
 convolution  integral in the lemma is constant on $\T$-cosets in $G_\sigma$. Therefore, the integral makes sense. 
Now, if $\tilde{a}:G\to\A_\sigma$ is given by $\tilde{a}_g=[a_{\tig}, \tig]$ for $\tig\in G_\sigma$ with $q(\tig)=g$,
and similarly for $b$, 
we compute
\begin{align*}
(\tilde{a}*\tilde{b})_g&=\int_G \tilde{a}_h\tilde{b}_{h^{-1}g}\,\dd h
= \int_G [a_{\tilde{h}}, \tilde{h}][b_{\tilde{h}^{-1}\tig}, \tilde{h}^{-1}\tig]\,\dd h\\
&=\int_G [a_{\tilde{h}}b_{\tilde{h}^{-1}\tig},\tig]\,\dd h
=[(a*b)_{\tig},\tig],
\end{align*}
which settles the claim on the convolution. We leave it as an exercise to check that the involution is preserved as well.

For the final assertion we use the fact that by Gleason's theorem \cite{Gleason}*{Theorem 4.1} there always exist local continuous section for the quotient map $q:G_\sigma\to G$. Thus, given an element $a\in C_c(G_\sigma,\A,\iota)$, we can find a finite open cover
$U_1,\ldots, U_l$ of $K:=q(\supp a)\subseteq G$ together with continuous maps $\s_i:U_i\to q^{-1}(U_i)$ such that 
$q\circ \s_i=\id_{U_i}$ for all $1\leq i\leq l$.  Let $\{\chi_i: 1\leq i\leq l\}\subseteq C_c^+(G)$ be a partition of the unit of $K$ subordinate to $\{U_i: 1\leq i\leq l\}$.
For each $\tig\in q^{-1}(U_i)$  there exists a unique element $z_{\tig,i}\in \T$ 
such that  $\tig=z_{\tig,i}\s_i(g)$, with $g=q(\tig)$. We use this to define $f_i:G_\sigma\to \C$ by $f_i(\tig)=\bar{z}_{\tig,i}\sqrt{\chi_i(g)}$ for $\tig=z_{\tig,i}\s_i(g)\in q^{-1}(U_i)$ and 
$0$ else. Moreover, we define $\tilde{a}_i\in C_c(\A)$ by $\tilde{a}_i(g)=\sqrt{\chi_i(g)}a_{\s_i(g)}$ for all $g\in U_i$ and $0$ else.
Then for all $\tig\in G_\sigma$ we get with $g=q(\tig)$:
\begin{align*}
a({\tig})&=\sum_{i=1}^l\chi_i(g)a({\tig})=\sum_{g\in U_i}\chi_i(g)a({z_{\tig,i}\s_i(g)})\\
&=\sum_{g\in U_i} \bar{z}_{\tig,i}\sqrt{\chi_i(g)}\sqrt{\chi_i(g)}a({\s_i(g)})
=\sum_{i=1}^l f_i(\tig)\tilde{a}_i(g)
\end{align*}
and the result follows.
\end{proof}

In the special case where the extension $\sigma=(\T\into G_\sigma\stackrel{q}{\onto} G)$ 
admits a continuous section $\s:G\to G_\sigma$ for the quotient map $q$ with $\s(e)=e$, the twisted Fell bundle $\A_\sigma$ has a much more direct description as follows: let $\om=\partial \s\in Z^2(G,\T)$ denote the cocycle corresponding to $\sigma$ and $\s$.
Since $\s$ is continuous, $\om$ is continuous as well. We then may define an $\om$-twisted multiplication on the Banach bundle $\A$ 
by the formulas
    \begin{equation}\label{eq:twisted-operations-Fell-bundle}
        a_g\cdot_\omega a_h:=\omega(g,h)a_ga_h,\quad a_g^{*,\omega}:=\overline\omega(g^{-1},g)a_g^*,\quad a_g\in A_g,a_h\in A_h.
    \end{equation}
This gives a new Fell bundle structure on the Banach bundle $\A$ which we denote by $\A_\om$.
For the following proposition recall that for any fixed element $\tig\in q^{-1}(G)\subseteq G_\sigma$ 
the elements in the fibre $A_{\sigma,g}$ have a unique representative of the form $(a_g, \tig)\in q^*\A$. 
In particular, if we apply this to $\tig=\s(g)$ we can identify $A_g$ with $A_{\sigma,g}$ via $a_g\mapsto [a_g, \s(g)]$.

\begin{proposition}\label{prop-om-Fellbundle}
Let $\A$, $\sigma$ and $\om$ be as above. Then the map
$$\Theta: \A_\om\to \A_\sigma: a_g\mapsto [a_g, \s(g)]$$
is an isomorphism of Fell bundles.
\end{proposition}
\begin{proof} As observed above, the map $\Theta$ is a bijection on each fibre, so we only need to see that it is a homeomorphism and preserves multiplication and involution. Indeed, arguments similar to those given in the proof of Lemma \ref{lem:eq-omsections} show that $\Theta$ is a homeomorphism.

To check that $\Theta$ is multiplicative let $[a_g, \s(g)]\in A_{\sigma, g}$ and  $[b_h, \s(h)]\in A_{\sigma,h}$. Then
\begin{align*}
&\Theta(a_g\cdot_{\om}b_h)=\Theta(\overline{\om(g,h)}a_gb_h)\\
&=[\overline{\om(g,h)}a_gb_h, \s(gh)]=[a_gb_h, \om(g,h)\s(gh)]\\
&=[a_gb_h, \s(g)\s(h)]=[a_g, \s(g)][b_h, \s(h)].
\end{align*} 
A similar computation shows that $\Theta$ also preserves involution.
\end{proof}

In what follows, if $\A$ is a Fell bundle over $G$, and $\sigma=(\T\into G_\sigma\stackrel{q}\onto G)$ is a twist for $G$, then 
we want to prove that  the algebras of kernels $\kk\A$ and $\kk{\A_\sigma}$ are isomorphic.
Recall that $\kk\A\cong C^*(\A)\rtimes_{\delta_\A}\dualG$ and that under this isomorphism the 
structural homomorphism $j_{C_0(G)}$ and the dual action $\dual{\delta_\A}$ are given as
$$\big(j_{C_0(G)}(f)k\big)(g,h)=f(g)k(g,h) \quad \text{and}\quad \big(\dual{\delta_\A}(g)k\big)(s,t)=\Delta(g)k(sg,tg).$$
Thus the dual weak $G\rtimes G$-algebra $\big(C^*(\A)\rtimes_{\delta_\A}\dualG, \dual{\delta_\A}, j_{C_0(G)}\big)$ is given as
$$
(B,\beta,\phi)=(\kk\A, \dual{\delta_\A}, j_{C_0(G)}).
$$
Similarly, the dual weak $G\rtimes G$-algebra for $(C^*(\A_\sigma), \delta_{\A_\sigma})$ is given by the triple
$$(C,\gamma,\psi)=(\kk{\A_\sigma}, \dual{\delta_{\A_\sigma}}, j^\sigma_{C_0(G)}).$$
We aim to show that the $\sigma$-twisted weak $G\rtimes G$-algebra $(B_\sigma,\beta_\sigma,\phi_\sigma)$ of Definition \ref{def-cosystem} 
is isomorphic to 
$(C,\gamma,\psi)$.
This will then show that for $(A,\delta)=(C_\mu^*(\A), \delta_\mu)$ and any duality crossed-product functor $\rtimes_\mu$ the 
deformed cosystem $(A^\sigma, \delta^\sigma)$ is isomorphic to $(C_\mu^*(\A_\sigma), \delta^\sigma_\mu)$.

We start with a description of $\kc{\A_\sigma}$ and $\kb{\A_\sigma}$ as follows:

\begin{lemma}\label{lem-kcAsigma}
The algebra  $\kc{\A_\sigma}$ can be identified with  the space of  continuous
compactly supported functions 
$k: G_\sigma\times G_\sigma\to \A$ satisfying the relations
\begin{equation}\label{eq-relation}
k(z\tilde g, u\tilde h)=\bar{z}u k(\tilde g, \tilde h)\in A_{gh^{-1}}\quad \forall z,u\in \T, \tig \in q^{-1}(g), \tih\in q^{-1}(h).
\end{equation}
Under this identification, convolution and involution are  given by the formulas 
\begin{equation}\label{eq-convinvFell}
\begin{split}
k*l(\tig, \tih)&=\int_{G} k(\tig,\tir)\cdot l(\tir,\tih)\dd r\quad\text{and}\\
k^*(\tig, \tih)&=k(\tih, \tig)^*.
\end{split}
\end{equation}
where $r=q(\tir)$ and  the product $\cdot$ and involution $^*$ are taken from $\A$.
Similarly, the algebra $\kb{\A_\sigma}$ can be identified with bounded measurable compactly supported 
functions satisfying \eqref{eq-relation}. 
\end{lemma}

Note  that it follows from \eqref{eq-relation} that the integrand in the convolution integral only depends on 
$r=q(\tir)$, so the formula makes sense. 

\begin{proof} 
The proof can be done similarly to the proof of Lemma~\ref{lem:eq-omsections}: if 
$\tilde{k}\in \kc{\A_\sigma}$, then for $(t,g)\in G\times G$ and $(\tig,\tih)\in q^{-1}(g)\times q^{-1}(h)$ 
we have $\tilde{k}(g,h)=A_{\sigma, gh^{-1}}$. Hence,  there is a unique element $k(\tig, \tih)\in A_{gh^{-1}}$ such that 
$\tilde{k}(g,h)=[k(\tig,\tih), \tig\tih^{-1}]$. Exactly as in the proof of  Lemma~\ref{lem:eq-omsections} we can check that 
$\tilde{k}\leftrightarrow k$ gives a bijection between sections of $\A_\sigma$ and functions as in~\eqref{eq-relation}, and that $\tilde{k}$ is continuous (with compact supports) if and only if $k$ is. It is then straightforward to check the formulas for convolution and involution as in~\eqref{eq-convinvFell}.
\end{proof}

\begin{remark}\label{rem-formula}
Using Lemma~\ref{lem:eq-omsections},
it follows from~\eqref{eq-kernel} that the kernel $k_{a,f}$ for $a\in C_c(\A_\sigma)$ and $f\in C_c(G)$ can be 
identified with the function $k_{a,f}:G_\sigma\times G_\sigma\to \A$ given by
\begin{equation}\label{eq-kaf} k_{a,f}(\tilde g,\tilde h)= a\big(\tilde g \tilde h^{-1}\big)f(g)\Delta(g)^{-1}.
\end{equation}
Since $(z\tilde g)(u\tilde h)^{-1}=z\overline u \tilde g\tilde h^{-1}$ and $a(z\tilde g)=\overline z a(\tilde g)\in A_g$,
these kernels satisfy~\eqref{eq-relation}. 
\end{remark}

From now on, we shall always identify 
$\kc{\A_\sigma}$ with the set of functions $k:G_\sigma\times G_\sigma\to\A$ as in Lemma \ref{lem-kcAsigma} above. 
In the following proposition let $\s:G\to G_\sigma$ be any Borel cross-section for the quotient map $q:G_\sigma\to G$, and let
$\om=\partial\s\in Z^2(G,\T)$ be the associated Borel cocycle. 
\begin{proposition}\label{prop-Fellisoom}
Given a twist $\sigma=(\T\into G_\sigma\onto G)$ and a  
Fell bundle $\A$ over $G$, with associated twisted Fell bundle $\A_\sigma$, we have an isomorphism of kernel \cstar{}algebras
\begin{equation}\label{eq:iso-kernels-algebras}
\Phi\colon \kk{\A_\sigma}\congto\kk{\A},\quad k\mapsto \Phi(k)(g,h):=k(\s(g),\s(h)).
\end{equation}
The isomorphism $\Phi$ commutes with the canonical structural homomorphisms from 
$\contz(G)$ and sends the dual $G$-action $\dual{\delta_{\A_\sigma}}$ on $\kk{\A_\sigma}$ to the $G$-action $\dual\delta^\omega$ on $\kk{\A}$ given by
\begin{equation}\label{eq:twisted-dual-action}
\dual\delta^\omega_t(k)(g,h):=\Delta(t)\overline{\om(g,t)}\om(h,t)k(gt,ht)
\end{equation}
for all $k\in \kb{\A}$, $t,g,h\in G$.
\end{proposition}
\begin{proof}
The discussion preceding Proposition~\ref{prop-kernels} shows that we can regard $\kk{\A}$ also  as the enveloping \cstar{}algebra of the Banach \Star{}algebra $\ktwo{\A}$, the  completion of the space $\kb{\A}$ 
with respect to the $L^2$-norm
    $\|k\|_2=\left(\int_{G\times G} \| k(g,h)\|^2\,\dd g\dd h\right)^{\frac{1}{2}}$.
Now given the Borel section $\s\colon G\to G_\sigma$, we define
$$\Phi: \kb{\A_\sigma}\to\kb{\A};\quad  \Phi(k)(g,h)=k(\s(g), \s(h)).$$ 
Straightforward computations show that this map preserves convolution and involution, and is isometric for $\|\cdot\|_2$. 
To see that it is surjective, observe that for every $l\in \kb\A$ we can define $k\in \kb{\A_\sigma}$ by $k(\s(g)z, \s(h)u):=\bar{z}ul(g,h)$, and then $l=\Phi(k)$.

It follows that $\Phi$ extends to an isometric isomorphism between $\ktwo{\A_\sigma}$ and $\ktwo\A$, and hence to their enveloping \cstar{}algebras $\kk{\A_\sigma}$ and $\kk{\A}$. It follows from~\eqref{eq-jGk} that this isomorphism preserves the structure maps $j_{\contz(G)}$ and $j_{\contz(G)}^\sigma$. Finally, for the dual actions we compute 
\begin{multline*}
\Phi(\dual{\delta^\sigma}_t(k))(g,h)=\dual{\delta^\sigma}_t(\s(g),\s(h))=\Delta(t)k(\s(g)\s(t),\s(h)\s(t))\\
=\Delta(t)k(\omega(g,t)\s(gt),\omega(h,t)\s(h,t))=\Delta(t)\overline{\om(g,t)}\om(h,t)k(gt,ht).
\end{multline*}
This yields the final assertion involving the dual actions.
\end{proof}

The above proposition  provides the main step of the proof of 
 
 \begin{theorem}\label{thm-fell-bundle-om}
 Let $\A\to G$ be a Fell bundle and let $\sigma=(\T\into G_\sigma\onto G)$ be a twist over $G$. 
 Then the deformation $(B_\sigma, \beta_\sigma, \phi_\sigma)$  of the weak $G\rtimes G$-algebra $(B, \beta,\phi)=(\kk{\A}, \dual{\delta_\A}, j_{C_0(G)})$ is isomorphic to 
 $(\kk{\A_\sigma}, \dual{\delta_{\A_\sigma}}, j_{C_0(G)}^\sigma)$.
 Hence, for every  duality crossed-product functor $\rtimes_\mu$ for $G$ the dual $G$-coaction $\big(C^*_\mu(\A_\sigma), \delta^\sigma_\mu)$ is isomorphic to the deformation
 $(A^\sigma, \delta^\sigma)$ of  $(A,\delta)=(C_\mu^*(\A), \delta_\mu)$.
 \end{theorem}
 \begin{proof}
We know from Corollary~\ref{cor:structural-hom-ext} that the structural homomorphism $\phi\colon \contz(G)\to \M(\kk\A)$ extends to $\bar\phi\colon L^\infty(G)\to \M(\kk\A)$. It follows 
from \cite{BE:deformation}*{Proposition~5.15} that the deformed \cstar{}algebra $B_\sigma=\kk{\A}_\sigma$ is isomorphic to $B=\kk{\A}$, which is also isomorphic to $\kk{\A_\sigma}$ by Proposition~\ref{prop-Fellisoom}.  All these isomorphisms preserve the structural homomorphism from $\contz(G)$, and the dual $G$-action on $\kk{\A_\sigma}$ is sent to the action $\dual\delta^\omega\colon G\car \kk{\A}$ given by~\eqref{eq:twisted-dual-action}. This also equals the $G$-action $\Ad U_\omega(t)\circ\dual\delta$, where $U_\omega(t)=\bar\phi(u_\omega(t))$, with $u_\omega(t)(g)=\overline{\omega(g,t)}$. 
This is precisely the action that appears in \cite{BE:deformation}*{Proposition~5.15}. It follows that $(B_\sigma, \beta_\sigma, \phi_\sigma)\cong (\kk{\A_\sigma}, \dual{\delta_{\A_\sigma}}, j_{C_0(G)}^\sigma)$.
The final assertion follows from Proposition \ref{prop-mu-cross-sectional-algebra} in combination with Proposition~\ref{prop-iso-fixed}.
 \end{proof}

For use in the next section, we now want to give an alternative description of the isomorphism $(B_\sigma,\beta_\sigma,\phi_\sigma)\cong (\kk{\A_\sigma}, \dual{\delta_{\A_\sigma}}, j_{C_0(G)}^\sigma)$ of the above theorem, which does not depend on a choice of a Borel section $\s:G\to G_\sigma$. The need for this comes from the fact that in general, if 
$$\Sigma=(X\times \T\into \mathcal G_\Sigma\onto X\times G)$$
is a continuous family of twists $\sigma_x=(\T\into G_{\sigma_x}\onto G)$, $x\in X$, as introduced in  \cite{BE:deformation} to study continuity properties of our deformation process, we do not know whether there exists a global Borel cross-section $\mathfrak{S}:X\times G\to \G_\Sigma$ which induces Borel cross-sections $\s_x:G\to G_{\sigma_x}$ in the fibres.

To overcome this problem, if $(B,\beta,\phi)=(\kk\A, \widehat{\delta_\A}, j_{C_0(G)})$ as above, recall the 
Hilbert $B$-module 
$$\L(G_\sigma,B)=\L(G_\sigma,\kk{\A})=C_0(G_\sigma,\iota)\otimes_{C_0(G)}\kk{\A}.$$
Given a section $\s:G\to G_\sigma$, it follows from Corollary~\ref{cor:structural-hom-ext},  \cite{BE:deformation}*{Proposition~5.15} and~\eqref{eq-jGk} that 
$\L(G_\sigma,B)\cong B$ as Hilbert $B$-modules with
isomorphism given by 
\begin{equation}\label{eq:isomorphism-L(G,B)=B}
\Theta\colon \L(G_\sigma,\kk{\A})\congto \kk{\A}, \quad f\otimes k\mapsto \Theta(f\otimes k)(g,h):=f(\s(g))k(g,h).
\end{equation}
In order to  provide an alternative description of $\L(G_\sigma,B)$, we let $\X_c(\A)$ denote the set of compactly supported continuous functions 
$$\xi: G_\sigma\times G\to \A\quad\text{with} \quad\xi(\tilde g, h)\in A_{gh^{-1}}\quad\text{and}\quad \xi(z\tilde g, h)=\bar{z}\xi(\tilde g, h),$$
for all $z\in \T$, $\tilde g\in G_\sigma, h\in G$ with $g=q(\tilde g)$. Define a $\kc{\A}$-valued inner product on $\X_c(\A)$ by
\begin{equation}\label{eq-rinner}
\braket{\xi}{\eta}_{\kc{\A}}(g,h)= \int_{G} \xi(\tir, g)^*\eta(\tir, h)\dd r.
\end{equation}
with $r=q(\tir)$. As before, one checks that the integrand only depends on $r=q(\tir)\in G$, so the integral makes sense.
On the other side, we have a  $\kc{\A_\sigma}$-valued left inner product on $\X_c(\A)$  given by the formula
$${_{\kc{\A_\sigma}}\braket{\xi}{\eta}}(\tig, \tih)=\int_G \xi(\tig, r)\eta(\tih, r)^*\dd r.$$
One easily checks that ${_{\kc{\A_\sigma}}\lk \xi, \eta\rk}$ satisfies~\eqref{eq-relation}. Moreover, we have
left and right actions of $\kc{\A_\sigma}$ and $\kc{\A}$ on $\X_c(\A)$ 
given on kernels by
\begin{align*}
k\cdot\xi(\tig, t)&=\int_{G}k(\tig, \tih)\xi(\tih, t)\dd h\quad\text{and}\\
\xi\cdot l(\tig, t)&=\int_{G}\xi(\tig, r) l(r,t)\dd r.
\end{align*}
One checks that these formulas satisfy the usual algebraic compatibility conditions for a pre-equivalence bimodule, as for instance the relation $_{\kc{\A_\sigma}}\braket{\xi}{\eta}\cdot \zeta=\xi\cdot\braket{\eta}{\zeta}_{\kc\A}$ for all $\xi,\eta,\zeta\in \X_c(\A)$. Finally we define an action  $\gamma:G_\sigma\car \chi_c(\A)$ by 
\begin{equation}\label{eq-actionchicA}
( \gamma_{\tig}\xi)(\tih, r)={\Delta(g)}\xi(\tih\tig, rg)\quad \text{with $g=q(\tig)$.}
\end{equation}
An easy computation shows that it is compatible with (the inflations to $G_\sigma$ of) the dual actions $\dual{\delta_{\A_\sigma}}$ on $\kc{\A_\sigma}$ on the left and the Green-twisted action $(\dual{\delta_\A},\iota^\sigma)$ on $\kk\A$ on the right.

\begin{proposition}\label{prop-iso-bimodules}
The $\kc{\A_\sigma}-\kc\A$ pre-equivalence bimodule $\X_c(\A)$ completes to give a 
$\dual{\delta_{\A_\sigma}}-(\dual{\delta_\A}, \iota^\sigma)$-equivariant $\kk{\A_\sigma}-\kk\A$-equivalence bimodule $(\X(\A), \gamma)$. 
 Moreover, the map
\begin{equation}\label{eq-iso-modules}
\Psi: C_0(G_\sigma,\iota)\odot \kc{\A}\to \X_c(\A);\quad \Psi(f\otimes k)(\tig, t):= f(\tig)k(q(\tig),t)
\end{equation}
extends to an $(\rt^{\sigma}\otimes \dual{\delta_\A})-\gamma$ equivariant
 isomorphism of right Hilbert $\kk{\A}$-modules $\L(G_\sigma,\kk{\A})\congto \X(\A)$ which then induces an isomorphism
of  weak $G\rtimes G$-algebras 
$(B_\sigma,\beta_\sigma,\phi_\sigma)\cong (\kk{\A_\sigma}, \dual{\delta_{\A_\sigma}}, j_{C_0(G)}^\sigma)$ 
for $(B,\beta,\phi):=(\kk\A, \widehat{\delta_\A}, j_{C_0(G)})$.
\end{proposition}
\begin{proof} We use the second assertion for the proof of the first. 
Indeed, it is  straightforward to check that $\Psi$ preserves the right inner product and actions
 and an argument as in the proof of Lemma~\ref{lem-crosssectional}, using continuous local sections and partitions of the unit, shows that it is also surjective. 
 This implies that $\X_c(\A)$ is a  right pre-Hilbert $\kk\A$-module and that $\Psi$ extends to an isomorphism $\Psi: \L(G_\sigma,\kk\A)\congto \X(\A)$ of right Hilbert  $\kk\A$-modules. 
 By the compatibility conditions of the pairings for $\X_c(\A)$ we further see that $B_\sigma:=\kk\A_\sigma=\K(\L(G_\sigma,\kk\A))\cong\K(\X(\A))$ is a \cstar{}completion of $\kc{\A_\sigma}$. We need to show that it coincides with $\kk{\A_\sigma}$. 
For this it suffices to show that the left action of $\kc{\A_\sigma}$ on $\X_c(\A)$ extends faithfully  to a left action of $\kk{\A_\sigma}$ on $\X(\A)$.
For this we choose a Borel section $\s:G\to G_\sigma$ and recall the isomorphism $\Theta: \L(G_\sigma, \kk\A)\congto \kk\A$ as right Hilbert $\kk\A$-modules 
 as in~\eqref{eq:isomorphism-L(G,B)=B}. Then the 
the composition of isomorphisms
$$\X(\A)\stackrel{\Psi^{-1}}\congto\L(G_\sigma,\kk{\A})\stackrel{\Theta}\congto \kk{\A}$$
sends a function $\xi\in \X_c(\A)$ to the function $\tilde\xi\in \kc{\A}$ 
given by $\tilde\xi(g,h):=\xi(\s(g),h)$. A simple computation then shows that
$$k\cdot \xi =\Phi(k)\cdot \tilde\xi,\quad \forall k\in \kc{\A_\sigma}, \xi\in \X_c(\A),$$
where $\Phi:\kk{\A_\sigma}\congto \kk\A$ is the isomorphism of~\eqref{eq:iso-kernels-algebras}.
This proves that the left action is bounded and extends faithfully to a left action of $\kk{\A_\sigma}$ on  $\X(\A)$. To complete the proof, we 
need to check that  $\Psi$ is $\gamma-\rt^\sigma\otimes\widehat{\delta_\A}$ equivariant and that it intertwines the left structure maps 
$j_{C_0(G)}^\sigma:C_0(G)\to \Lb_{\kk\A}(\chi(\A))$ with the left action  $\phi_\sigma:C_0(G)\to \Lb_{\kk\A}(\L(G_\sigma,\kk\A)$ induced  from the 
left action of $C_0(G)$ on $C_0(G_\sigma,\iota)$ by pointwise multiplication.  We  do the first and leave the second to the reader:
for all $f\in C_0(G_\sigma)$ and $k\in \kc\A$ we compute
\begin{align*}
\Psi\Big((\rt^{\sigma}\otimes\dual{\delta_\A})_{\tig}(f\otimes k)\Big)(\tih, r)&=
\Psi\big(\rt^\sigma_{\tig}(f)\otimes \dual{\delta_\A}_{\tig}(k)\big)(\tih,r)\\
&=f(\tih\tig)\Delta(g)k(hg,rg)\\
&=\Delta(g)(f(\tih\tig)k(hg, rg))\\
&=\gamma_{\tig}\big(\Psi(f\otimes k)\big)(\tih,r)
\end{align*}
for all $\tig,\tih\in G_\sigma$, $r\in G$ such that $g=q(\tig), h=q(\tih)$. 
This  finishes the proof.
\end{proof}

\section{Continuous fields}\label{sec-cont}

In this section we  study  continuity properties of deformation of cross-sectional algebras of Fell bundles with respect to a continuous family of twists.
Our results extend  results of Raeburn in \cite{Raeburn:Deformations} who considered deformation of Fell bundles by continuous families of circle valued $2$-cocycles of discrete amenable groups $G$. Recall from \cite{BE:deformation} that by a {\em continuous family} of twists of $G$ over a locally compact Hausdorff space $X$,  or simply a {\em twist over} $X\times G$,
 we  understand a groupoid central extension
\begin{equation}\label{eq-groupoid}
\Sigma:=(X\times \T\stackrel{\iota}{\into} \G\stackrel{q}\onto X\times G)
\end{equation}
of the trivial group bundle $X\times G$ by the central  trivial group bundle $X\times \T$. Then for each $x\in X$, we obtain a central extension $\sigma_x:=(\T\into G_{\sigma_x}\onto G)$ of $G$ by $\T$. For ease of notation, we shall denote the elements in $\G$ generally by $\tilde g$, and by $(x,\tilde g)$ to indicate that $\tilde g$ lies in the fibre $G_{\sigma_x}=q^{-1}(\{x\}\times G)\sbe \G$ over $x\in X$, and we shall write $(x, g)$ for its image in $X\times G$ under the quotient map. 

As explained in \cite{BE:deformation}, continuous families of twists of $G$ over $X$ are closely related to $\contz(X)$-linear actions on continuous trace \cstar{}algebras. More precisely, the following construction provides natural examples of continuous families of twists:

\begin{example}\label{eq-action}
Let $\K=\K(\H)$ for a Hilbert space $\H$ and suppose that $\alpha:G\car C_0(X,\K)$ is a $C_0(X)$-linear (i.e.~ fibre-wise) action of $G$ on $C_0(X,\K)$ as considered in \cite{BE:deformation}*{Definition~4.23}.
For all $x\in X$ let $\alpha^x:G\car \K$ denote the action on the fibre at $x$.
Then $$\Sigma_{\alpha}:=\{(x,g,v)\in X\times G\times \U(\H): \alpha_g^x=\Ad v\}$$
together with  the embedding 
$X\times \T\into \G_\alpha; (x,z)\mapsto (x,e, \bar{z}1_{\H})$ and the quotient map $\G_\alpha\onto X\times G; 
(x,g,v)\mapsto (x,g)$ defines a groupoid central extension
$$\Sigma_\alpha=(X\times \T\into \G_\alpha\onto X\times G)$$
as above. We refer to \cite{BE:deformation}*{Lemma~4.28} for further details of this construction. 
\end{example}

\begin{remark}\label{rem-L2Sigma}
We also need to recall that, conversely, every twist $\Sigma=(X\times \T\stackrel{\iota}{\into} \G\stackrel{q}\onto X\times G)$ gives rise to 
a fibre-wise action of $G$ on a continuous field of compact operators  over $X$.
If $G$ and $X$ are second countable, this precisely inverts the construction of the above example.
For the construction let $C_c(\G,\iota)$ denote the space of compactly supported continuous functions $\xi$ on $\G$ which satisfy the relation
\begin{equation}\label{eq-relationSigma}
\xi(x, \tig z)=\bar{z}\xi(x,\tig)\quad\forall (x,\tig)\in \G, z\in \T.
\end{equation}
We define a $C_0(X)$-valued inner product on $C_c(\G,\iota)$ by
\begin{equation}\label{eq-C0(X)innerproduct}
\braket{\xi}{\eta}_{C_0(X)}(x)=\int_G \overline{\xi(x,\tig)}\eta(x,\tig)\dd g
\end{equation}
and we let $L^2(\G,\iota)$ denote the completion of $C_c(\G,\iota)$ with respect to this inner product. 
Then the right translation action $\rho_\G:\G\car L^2(\G, \iota)$ is given fibre-wise 
by the right translation action $\rho^{\sigma_x}:G\car L^2(G_{\sigma_x}, \iota)$ given by  $(\rho^{\sigma_x}_\tig\xi)(\tih)=\sqrt{\Delta(g)}\xi(\tih\tig)$
for $\tig,\tih\in G_{\sigma_x}$ and $g=q_x(\tig)$.  The adjoint action $\Ad\rho_\G$ then provides a $C_0(X)$-linear action 
$\alpha:G\car \K(L^2(\G,\iota))$ with fibre actions $\alpha^x=\Ad\rho^{\sigma_x}:G\car \K(L^2(G_{\sigma_x}, \iota))$.
It has been shown in  \cite{BE:deformation}*{Lemma 4.33} that if $\Sigma=\Sigma_{\alpha'}$ for some 
continuous family of actions $\alpha':G\car C_0(X,\K)$ as in the above example, then $( \K(L^2(\G,\iota)), \alpha)$ and $(C_0(X,\K),\alpha')$
are $X\times G$-equivariantly Morita equivalent.
\end{remark}

Another source of examples for continuous families of twists is given via continuous families of $2$-cocycles:

\begin{example}\label{eq-cocycles} Suppose $G$ is a second countable locally compact group.  
By a continuous family of Borel $2$-cocycles $x\mapsto \om_x$ over the second countable locally compact space $X$ we understand a 
Borel $2$-cocycle $\Omega\colon G\times G\to C(X,\T)$, where the trivial $G$-module $C(X,\T)$ is equipped with the 
topology of uniform convergence on compact subsets of $X$, such that $\om_x:=\Om(\cdot,\cdot)(x)\in Z^2(G,\T)$ for all $x\in X$. 
It follows then from \cite{HORR}*{Proposition 3.1} together with \cite{CKRW}*{Theorem 5.1(3)} that $\Omega$ induces 
a $C_0(X)$-linear action $\alpha:G\car C_0(X,\K(L^2(G))$ by defining $\alpha^x:=\Ad \rho^{\bar\om_x}$ for all $x\in X$,
where $\rho^{\bar\om_x} :G\to \U(L^2(G))$ denotes the $\bar\om_x$-right regular representation 
of $G$ (see \cite{BE:deformation}*{Remark 4.19}).
By the construction in Example \ref{eq-action} we then obtain a twist $\Sigma_{\Omega}:=\Sigma_\alpha$
and it follows from  \cite{BE:deformation}*{Theorem 4.14.}  that for all $x\in X$ the twist $\sigma_x=(\T\into G_x\onto G)$ at $x$ is 
isomorphic to the twist $\sigma_{\om_x}=(\T\into G_{\om_x}\onto G)$ corresponding to the cocycle $\om_x$.

Notice that if $\Om:G\times G\to C(X,\T)$ is a {\em continuous} cocycle (e.g., if $G$ is discrete), 
then $\Sigma_{\Om}$ can be constructed directly 
as  $\Sigma_{\Om}:=(X\times\T\into \G_{\Om}\onto X\times G)$ where  $\G_{\Om}=G\times X\times \T$ is equipped  with product topology,   multiplication defined by
$(g_1, x, z_1)(g_2, x, z_2)=(g_1g_2, x, \Om(x)(g_1, g_2) z_1z_2)$, and the obvious inclusion of $X\times \T$ and projection onto $X\times G$. 
This construction does not need any second countability assumptions. If $\Om$ is not continuous, it is not clear to us how to topologize $G\times X\times \T$ in this construction.
\end{example}

A third class of important examples comes from central group extensions. The following example is \cite{BE:deformation}*{Proposition 4.35}.

\begin{example}\label{ex-extension}
Let $Z\stackrel{\iota_Z}{\into} H\stackrel{q_H}{\onto} G$ be a central extension of $G$ by the abelian group $Z$.
Then $Z$ acts freely and properly on the product space $\widehat{Z}\times H\times \T$ by
$$z(\chi, h, w):= (\chi, zh, \chi(z)w)\quad\forall z\in Z, (\chi, h,w)\in \widehat{Z}\times H\times  \T$$
and  there is a twist 
$$\Sigma_H:=(\widehat{Z}\stackrel{\iota}{\times} \T\into \G_H\stackrel{q}{\onto} \widehat{Z}\times G)$$
with $\G_H:=(\widehat{Z}\times H\times  \T)/Z$ and  inclusion and quotient maps are given by
$$\iota:(\chi, w)\mapsto [\chi, e, w]\quad{and}\quad q:[\chi, h, w]\mapsto (\chi, q_H(h)).$$
Here $e$ denotes the neutral element of $H$.

For each $\chi\in \widehat{Z}$ the fibre $\sigma_\chi=(\T\stackrel{\iota_\chi}\into G_\chi\stackrel{q_\chi}\onto G)$ at $\chi$ is then 
given by $G_\chi:=(H\times \T)/Z$ with respect to the action $z(h,w)=(zh, \chi(z)w)$ for $z\in Z, (h,w)\in H\times \T$.
The inclusion and quotient maps are given by
$\iota_\chi:w\mapsto [e,w]$ and  $q_\chi:[h,w]\mapsto q_H(h)$, respectively.
\end{example}

\begin{notation}\label{not-smooth}
   If the locally compact group $G$ admits a central extension $Z\stackrel{\iota_Z}{\into} H\stackrel{q_H}{\onto} G$ as above such that 
   the {\em transgression map} $\tg:\widehat{Z}\to \Twist(G); \chi\mapsto [\sigma_\chi]$ is bijective (hence an isomorphism of groups), we
   call $Z\stackrel{\iota_Z}{\into} H\stackrel{q_H}{\onto} G$ a {\em representation group} for $G$. 
   We say $G$ is {\em smooth}, if such a representation group exists.
   \end{notation}

  This notation was introduced by Calvin Moore in \cite{MooreII} in case of second countable locally compact groups, where he uses a Borel cross-section 
   $\s:G\to H$ in order to define a transgression map $\tg:\widehat{Z}\to H^2(G,\T)$. In that case, the transgression map of Moore can be obtained from ours by composing with the isomorphism $\Twist(G)\cong H^2(G,\T)$. Note that our construction does not need a Borel section, which may not exist if $G$ is not second countable. Notice that  many (but by far not all) groups are smooth in the above sense, 
   among them all discrete groups,  all semisimple Lie groups, the group $\R^n$, and many more.
   We refer to \cite{BE:deformation}*{Section 4} for a more detailed discussion and for concrete examples.

\subsection{Deformation via continuous families of twists}\label{subsec-continuousdeformation}
Suppose  that $(B, \beta, \phi)$ is a weak $G\rtimes G$-algebra and that $\rtimes_\mu$ is a duality crossed-product functor for $G$. 
Given a twist $\Sigma=(X\times \T\into \G\onto X\times G)$  over $X\times G$ with fibres $\sigma_x=(\T\into G_x\onto G)$,  
our general deformation process of  \cite{BE:deformation}*{Section 6} (see  Section~\ref{sec-Boreltwist} above) provides for all $x\in X$  the 
 weak $G\rtimes G$-algebras $(B_{\sigma_x},\beta_{\sigma_x},\phi_{\sigma_x})$ and then, via Landstad duality, the cosystems 
$(A^{\sigma_x}_\mu,\delta^{\sigma_x}_\mu)$. 
In \cite{BE:deformation}*{Theorem 6.16} we prove that the $(B_{\sigma_x},\beta_{\sigma_x},\phi_{\sigma_x})$ are  fibres of a continuous field
$(\B_\Sigma,\beta_\Sigma,\Phi_\Sigma)$
 of weak $G\rtimes G$-algebras over $X$, and from this we obtain, depending on certain properties of the given crossed-product functor $\rtimes_\mu$,
 continuity properties of the field of coactions $X\ni x\mapsto (A^{\sigma_x}_\mu,\delta^{\sigma_x}_\mu)$ over $X$.

We need to explain in more detail how $(\B_\Sigma,\beta_\Sigma,\Phi_\Sigma)$ is constructed and in which way it can be regarded as a continuous field of weak $G\rtimes G$-algebras. 
For this we consider the function space $\contz(\G,\iota)$ consisting of $\contz$-functions $\xi\colon \G\to \C$ satisfying~\eqref{eq-relationSigma}. 
Then $\contz(\G,\iota)$ becomes an imprimitivity Hilbert bimodule over $\contz(X\times G)$ with respect to the  inner products
\begin{equation}\label{eq-innerXxG}
{_{\contz(X\times G)}}\braket{\xi}{\eta}(x, g)=\xi(x,\tilde g)\overline{\eta(x,\tilde g)}\; \text{and}\; \braket{\xi}{\eta}_{\contz(X\times G)}(x, g)=\overline{\xi(x,\tilde g)}\eta(x,\tilde g)
\end{equation}
We then consider $\B:=\contz(X,B)$, viewed as a $\contz(X)$-algebra with constant fibre $B$. 
We define the (balanced) tensor product of Hilbert modules:
\begin{equation}\label{eq:mod-Fell-bundle}
    \E_\Sigma(\G,\B):=\big(\contz(\G,\iota)\otimes_{\contz(X\times G)}\B\big)\otimes_{\contz(X)}L^2(\G,\iota)^*,
\end{equation}
where $L^2(\G,\iota)^*$ denotes the $C_0(X)-\K(L^2(\G,\iota))$ equivalence bimodule dual to the right $C_0(X)$-Hilbert module  $L^2(\G,\iota)$  as constructed in Remark \ref{rem-L2Sigma}. 
It is shown in \cite{BE:deformation} (see the discussion preceding \cite{BE:deformation}*{Theorem 6.16}) that the 
$\G$-action
$$\gamma_{\G}:=\rt_{\G}\otimes_{C_0(X\times G)}\beta\otimes_{C_0(X\times G)} {\rho_\G}^*:\G\car  \E_\Sigma(C_0(X,B))$$
 for the right translation action  $\rt_\G:\G\car C_0(\G,\iota)$ and action $\rho_\G:\G\car L^2(\G,\iota)$ as in 
Remark \ref{rem-L2Sigma},  is trivial on $X\times \T$ and therefore factors through  a well-defined  
$C_0(X)$-linear action $\gamma_\Sigma:G\car \E_\Sigma(C_0(X,B))$.  Moreover, we denote by
$\Phi_\Sigma:C_0(X\times G)\to\M(\K(\E_\Sigma(\G, B)))$ the nondegenerate \Star{}homomorphism induced from the  left action  of $C_0(X\times G)$ on $C_0(\G,\iota)$.
We then define
$$(\B_\Sigma, \beta_\Sigma,\Phi_\Sigma):=(\K(\E_\Sigma(\G,\B)) ,\Ad \gamma_\Sigma, \Phi_\Sigma).$$
It is then easy to check that for all $x\in X$ the quotient maps $q_x:\B_\Sigma\to B_{\sigma_x}$ induce surjective morphisms of weak $G\rtimes G$-algebras
 $$q_x: (\B_\Sigma, \beta_\Sigma,\Phi_\Sigma)\onto (B_{\sigma_x},\beta_{\sigma_x}, \phi_{\sigma_x}).$$
Alternatively, we can also consider the right Hilbert $\B$-module
\begin{equation}\label{eq:mod-Fell-bundle1}
\L(\G,\B):=\contz(\G,\iota)\otimes_{\contz(X\times G)}\B
\end{equation}
equipped with the diagonal action $\epsilon_\Sigma:=\rt_{\G}\otimes_{C_0(X\times G)}\beta$ of  $\G$.
Since $L^2(\G,\iota)$ is an imprimitivity bimodule, we also have 
$$(\B_\Sigma, \beta_\Sigma,\Phi_\Sigma)\cong(\K(\L_\Sigma(\G,\B)) ,\Ad \epsilon_\Sigma, \Psi_\Sigma),$$
where, similar to $\Phi_\Sigma$, the \Star{}homomorphism $\Psi_\Sigma:C_0(X\times G)\to \M(\K(\L_\Sigma(\G,B)))$ is also induced from the left action of 
$C_0(X\times G)$ on $C_0(\G,\iota)$. 
We refer to  \cite{BE:deformation}*{Theorem~6.16} for more details.

\subsection{Deformation of Fell bundles via continuous families of twists}
We want to apply the above deformation procedure to the dual weak $G\rtimes G$-algebra 
$(B, \beta,\phi)=(C^*(\A)\rtimes_{\delta_\A}\dual G, \dual{\delta_\A}, j_{C_0(G)})$  of a Fell bundle $\A$ over $G$ 
and we want to show that, similar to deformation by a twist $\sigma=(\T\into G\onto G)$ as considered in the previous section, the construction of 
$(\B_\Sigma, \beta_\Sigma, \Phi_\Sigma)$ can then be done completely on the level of Fell bundles.
To explain this, given a Fell bundle $p:\A\to G$ and a twist $\Sigma=(X\times\T\into \G\onto X\times G)$, we 
construct a Fell bundle $\A_\Sigma$ over $X\times G$ by 
\begin{equation}\label{eq-ASigma}
\A_{\Sigma}=\big(\A\times_{(X\times G)}\G\big)/\sim
\end{equation}
where $\A\times_{(X\times G)}\G=\{(a, (x,\tig))\in \A\times \G: p(a)=q_x(\tig)\}$ and where $\sim$ denotes the equivalence relation
\begin{equation}\label{eq-sim}
(a, (x,\tig))\sim (b, (y,\tih))\Longleftrightarrow x=y\;\text{and}\; \exists z\in \T\;\text{such that}\; (b, (x,\tih))=(\bar{z}a, (x, z\tig)).
\end{equation}
The projection $p_\Sigma:\A_\Sigma\to X\times G$ is given by $p_\Sigma([a, (x,\tig)])=(x, g)$ if $g=q(\tig)$.
Observe that the restriction $\A_\Sigma|_{\{x\}\times G_{\sigma_x}}$ coincides with the  deformed Fell bundle $\A_{\sigma_x}$ 
with respect to the fibre $\sigma_x=(\T\into G_{\sigma_x}\onto G)$ at $x\in X$ as defined in the previous section.
As usual we write $C_c(\A_\Sigma)$ for the space of continuous sections $a:X\times G\to \A_\Sigma$ with compact supports.
Note that it becomes a \Star{}algebra with respect to the convolution
$$a*b(x,g)=\int_G a(x,h)b(x,h^{-1}g)\, \dd h\quad\text{and}\quad a^*(x,g)=\Delta(g^{-1})a_{(x,g^{-1})}^*$$
for $a,b\in C_c(\A_\Sigma)$.
The following lemma is then a complete analogue of Lemma \ref{lem-crosssectional} above and we omit the proof:

\begin{lemma}\label{lem-crossectionalSigma}
There is a bijection between the elements of $C_c(\A_\Sigma)$ and the set of compactly supported continuous functions
$a:\G\to \A$ satisfying
\begin{equation}\label{eq-relationSigma1}
a_{(x,z\tig)}=\bar{z}a_{(x,\tig)}\quad\forall (x,\tig)\in \G, z\in \T.
\end{equation}
Under this identification, convolution and involution are given by the formulas
\begin{equation}\label{eq-conv-invASigma}
a*b(x, \tig)=\int_G a(x,\tih)b(x,\tih^{-1}\tig)\, \dd h\quad\text{and}\quad a^*(x,\tig)=\Delta(g^{-1})a(x,\tig^{-1})^*
\end{equation}
where, as usual, we write $h=q_x(\tih)$, $g=q_x(\tig)$ for $q_x:G_{\sigma_x}\to G$.
\end{lemma}

There are constructions of full and reduced cross-sectional algebras of Fell bundles over groupoids, but 
in our situation these constructions can be reduced to the situation of Fell bundles over the group $G$ by
associating to  $\A_\Sigma$  a Fell bundle, say $\widetilde{\A}_\Sigma$ over $G$ as follows:
the  fibres  $\tilde{A}_g$ are given as the  $C_0$-sections  $C_0(\A_\Sigma|_{X\times \{g\}})$ of the restriction  $\A_\Sigma|_{X\times \{g\}}$ of $\A_\Sigma$ to $X\times \{g\}$
and multiplication and involution are defined pointwise over $X$.  

\begin{notation}\label{not-cross-sectional}
Let $\A_\Sigma$ and $\widetilde\A_\Sigma$ be as above and let $\rtimes_\mu$ be any duality crossed-product functor for $G$.
We then {\em write}
$C^*(\A_\Sigma):=C^*(\widetilde\A_\Sigma)$, and similarly $C_\mu^*(\A_\Sigma):=C^*_\mu(\widetilde{\A}_\Sigma)$. 
Then
 $C^*(\A_\Sigma)$ is equipped with a dual coaction  $\delta_{\A_\Sigma}: C^*(\A_\Sigma)\to \M(C^*(\A_\Sigma)\otimes C^*(G))$
 which factors through a coaction $\delta^\Sigma_\mu$ on $C_\mu^*(\A_\Sigma)$.
 \end{notation}

Note that there is an obvious inclusion 
$$C_c(\A_\Sigma)\into C_c(\widetilde{\A}_\Sigma); a \mapsto \big(g\mapsto a|_{X\times\{g\}})$$
 and one can show that this map induces an isomorphism of the usually defined full (or reduced) cross-sectional algebra $C^*(\A_\Sigma)$ (resp.~$C_r^*(\A_\Sigma)$) in the general setting of Fell bundles over groupoids and $C^*(\widetilde\A_\Sigma)$ (resp.~$C_r^*(\widetilde\A_\Sigma)$), so our definition makes sense; indeed the case of full cross-sectional \cstar{}algebras is a special case of \cite{Buss-Meyer:Groupoid_fibrations}*{Theorem~6.2} and the reduced case a special case of \cite{LaLonde}*{Proposition~5.1}. From now on, we shall simply identify $\A_\Sigma$ with the Fell bundle $\widetilde\A_\Sigma$ over $G$ whenever it seems convenient.
\medskip

The crossed product $C^*(\A_\Sigma)\rtimes_{\delta_{\A_\Sigma}}\dualG$ comes with the dual action $\dual{\delta_{\A_\Sigma}}$ of $G$ 
and the inclusion  $\Psi_G=j_{C_0(G)}:C_0(G)\to \M(C^*(\A_\Sigma)\rtimes_{\delta_{\A_\Sigma}}\dualG)$. We also have a
canonical nondegenerate \Star{}homomorphism $\Psi_X:C_0(X)\to Z\M(C^*(\A_\Sigma)\rtimes_{\delta_{\A_\Sigma}}\dualG)$
which is induced by pointwise multiplication $(\varphi a)(x,\tig)=\varphi(x)a(x,\tig)$ 
of functions $\varphi\in C_0(X)$ with sections  $a\in C_c(\A_{\Sigma})$. It is easily checked that it
 commutes with  $j_{C_0(G)}$. We therefore obtain a well-defined structure map
 \begin{equation}\label{eq-structue-map}
\Psi_{X\times G}:=\Psi_X\otimes \Psi_G:C_0(X\times G)\to \M(C^*(\A_\Sigma)\rtimes_{\delta_{\A_\Sigma}}\dualG).
  \end{equation}

\begin{theorem}\label{the:iso-continuous-deformation-Fell}
    Let  $(B, \beta,\phi)=(C^*(\A)\rtimes_{\delta_\A}\dual G, \dual{\delta_\A}, j_{C_0(G)})$ be the dual weak $G\rtimes G$-algebra
    for a Fell bundle $\A$ over $G$.
   Let $\Sigma=(X\times \T\into \G\onto X\times G)$ be a twist for $X\times G$ and let $(\B_\Sigma,\beta_\Sigma, \Phi_\Sigma)$ 
    be the deformation of $(B,\beta,\phi)$ by $\Sigma$ as in \S \ref{subsec-continuousdeformation}. 
    Then $(B_\Sigma,\beta_\Sigma, \Phi_\Sigma)$  is isomorphic to the triple 
    $$\big(C^*(\A_\Sigma)\rtimes_{\delta_{\A_\Sigma}}\dualG, \dual{\delta_{\A_\Sigma}},  \Psi_{X\times G}\big).$$
\end{theorem}
\begin{proof} 
Arguing similarly to the proof of Proposition \ref{prop-iso-bimodules}, we use the  descriptions of $B=C^*(\A)\rtimes_{\delta_\A}\dualG\cong \kk\A$ and 
$C^*(\A_\Sigma)\rtimes_{\delta_{\A_\Sigma}}\dualG\cong \kk{\A_\Sigma}$ in order to construct the desired isomorphism.
Let $\A_X=X\times \A$ be the pullback of $\A$ to $X\times G$ via the projection $X\times G\to G$ (this can be regarded as a special case of the construction of $\A_\Sigma$ for the 
trivial twist $X\times \T\into X\times (G\times \T)\onto X\times G$). We then have $C^*(\A_X)\cong C_0(X,C^*(\A))$ and
\begin{equation}\label{eq-AX}
\B:=C_0(X,B)= C_0(X,C^*(\A)\rtimes_{\delta_\A}\dualG) = C^*(\A_X)\rtimes_{\delta_{\A_X}}\dual G= \kk{\A_X}.
\end{equation}
We write $\kc{\A_\Sigma}$ for the compactly supported functions $k:\G\times_X\G\to \A$ satisfying
\begin{equation}\label{eq-relation-kcASigma}
k(x,z\tig,u\tih)=\bar{z}uk(x,\tig,\tih),\quad (x,\tig,\tih)\in \G\times_X\G,
\end{equation}
and observe that $\kc{\A_\Sigma}$ can be regarded as a  dense subalgebra of $\kk{\A_\Sigma}$ in a canonical way.

We follow  similar ideas as in the proof of Proposition \ref{prop-iso-bimodules}
and realize the right Hilbert $\B$-module $\L(\G,\B)$ of  \eqref{eq:mod-Fell-bundle1} as a completion of
the space $\X_c( \A_X)$ of all compactly supported continuous functions $\xi\colon \G\times G\to \A$ satisfying
$$\xi(x,z\tilde g,h)=\bar z\xi(x,\tilde g,h)\in A_{gh^{-1}}\mbox{ for all }(x,\tilde g)\in\G, h\in G.$$
We can mimic the formulas of the previous section and define inner products and left and right actions of $\kc{\A_\Sigma}$ and $\kc{\A_X}$ by the following formulas for $\xi,\eta\in \X_c(\A_X)$, $k\in \kc{\A_\Sigma}$ and $l\in \kc{\A_X}$:
$$\braket{\xi}{\eta}_{\kc{\A_X}}(x,s,t):=\int_{G_x}\xi(g,s)^*\eta(g,t)\dd g,$$
$$_{\kc{\A_\Sigma}}\braket{\xi}{\eta}(x,\tilde g,\tilde h):=\int_G\xi(x,\tilde g,t)\eta(x,\tilde h,t)^*\dd t,$$
$$(k\cdot\xi)(x,\tilde g,t):=\int_{G_x}k(x,\tilde g,\tilde h)\xi(x,\tilde h,t)\dd \tilde h,$$
$$(\xi\cdot l)(x,\tilde g,t):=\int_G\xi(x,\tilde g,h)l(h,t)\dd h.$$
The module $\L(\G,\B)$ is  $\contz(X)$-linear with fibres
$$\L(G_{\sigma_x},B)=\contz(G_{\sigma_x},\iota_x)\otimes_{\contz(G)}B,$$
where $\sigma_x=(\T\into G_{\sigma_x}\onto G)$ is the fibre of $\Sigma$ at $x\in X$.  
We  know from Proposition \ref{prop-iso-bimodules} 
that the algebra of compact operators of this module is isomorphic to $\kk{\A_{\sigma_x}}$, where $\A_{\sigma_x}$ is the Fell bundle deformed from $\A$ via $\sigma_x$.

We need to show  that  $\B_\Sigma=\K(\L(\G,\B))$ is isomorphic to $\kk{\A_\Sigma}$, and that this isomorphism intertwines the actions and structure maps.
Notice that $C^*(\A_\Sigma)$ is a $\contz(X)$-algebra with fibres $C^*(\A_{\sigma_x})$: this follows from the fact that full cross-sectional \cstar{}algebras of Fell bundles preserve exact sequences (see e.g. \cite{Exel:Book}*{Proposition~21.15} that proves this statement for Fell bundles over discrete groups; a similar proof applies for locally compact groups). 
 It follows then from \cite{Nilsen:Full}*{Theorem~4.3}  
that $\kk{\A_\Sigma}=C_0(\A_\Sigma)\rtimes_{\delta_{\A_\Sigma}}\dual G$ is also a $\contz(X)$-algebra with fibres
$$\kk{\A_{\sigma_x}}\cong C^*(\A_{\sigma_x})\rtimes_{\delta_x}\dualG.$$
Now, observe that we have a canonical map
$$\Psi_X\colon \contz(\G,\iota)\odot \kc{\A_X}\to \X_c(\A_X),\quad\Psi_X(f\otimes \xi)(x,\tilde g,h):=f(x,\tilde g)\xi(x, g,h).$$
As in the case of a single twist, one checks that this preserves the structures of right modules over $\kc{\A_X}$, so that $\Psi_X$ induces an isomorphism of right Hilbert $\kk{\A_X}$-modules 
$$\Psi_\Sigma\colon \L(\G,\B)=\contz(\G,\iota)\otimes_{\contz(X\times G)}\kc{\A_X}\congto \X(\A_X),$$
where $\X(\A_X)$ denotes the completion of $\X_c(\A_X)$ with respect to the right $\B$-module structure (recall that $\B=\kk{\A_X}$).
So, again, as in the case of a single twist, to show that $\kk{\A_\Sigma}=\B_\Sigma=\K(\L(\G,\B))$, it is enough to see that the left action of $\kc{\A_\Sigma}$ on $\X_c(\A_X)$ extends to a \Star{}homomorphism (i.e. a left action by adjointable operators)
$$\kk{\A_\Sigma}\to \Lb_{\B}(\L(\G,\B)).$$
But since both the algebra and the module involved carry $\contz(X)$-linear structures which are preserved by the left action of $\kc{\A_{\Sigma}}$ on 
$\X(\A_X)\cong  \L(\G,\B)$, 
 and we already know that the fibre-wise left action of $\kc{\A_{\sigma_x}}$ on the fibre $\L(G_{\sigma_x},B)$ of $\L(\G,\B)$ over $x$
  extends to an action by adjointable operators
$$\kk{\A_{\sigma_x}}\to\Lb_{B}(\L(G_x,B))$$
 the  result follows. 
 Finally notice that by Proposition  \ref{prop-iso-bimodules}  we know that the isomorphism $\kk{\A_{\sigma_x}}\to\Lb_{B}(\L(G_x,B))$ 
 intertwines the actions and structure maps and  therefore induces an isomorphism of weak $G\rtimes G$-algebras
 $$(B_{\sigma_x}, \beta_{\sigma_x}\phi_{\sigma_x})\cong (\kk{\A_{\sigma_x}}, \dual{\delta_{\A_{\sigma_x}}}, j_{C_0(G)}).$$
 Since the $C_0(X)$-linear actions and structure maps for $\B_\Sigma$ and $\kk{\A_\Sigma}$ induce these actions and structure maps on the fibres, 
 we can finally conclude the desired isomorphism  
 $$\big(B_\Sigma,\beta_\Sigma, \Phi_\Sigma\big)\cong\big(C^*(\A_\Sigma)\rtimes_{\delta_{\A_\Sigma}}\dualG, \dual{\delta_{\A_\Sigma}},  \Psi_{X\times G}\big).$$
 \end{proof}
 
 Recall that  $C^*(\A_\Sigma)$ 
  is a $C_0(X)$-algebra by extending pointwise multiplication 
 of functions in $C_0(X)$ with sections in $C_c(\A_\Sigma)$. The same holds true for $C^*_\mu(\A_\Sigma)$ for every  duality crossed-product 
 functor $\rtimes_\mu$. Recall that for any $C_0(X)$-algebra $D$, the (maximal) fibre $D_x$ of $D$ over $x$ is 
 defined as the quotient $D_x:=D/I_x$ with $I_x=C_0(X\setminus\{x\})D$. Thus we have the fibres $C_\mu^*(\A_\Sigma)_x$ 
 for each $x\in X$. 
 
On the other hand,   it is clear that evaluation at $x\in X$ induces $\delta^\Sigma_\mu-\delta^{\sigma_x}_\mu$ equivariant 
 quotient maps $Q_x:C^*_\mu(\A_\Sigma)\onto C^*_\mu(\A_{\sigma_x})$, and the obvious question arises, under 
 what conditions the $Q_x$ factor through isomorphisms $C_\mu^*(\A_\Sigma)_x \cong C^*_\mu(\A_{\sigma_x})$? By \cite{Dana:book}*{Theorem C.26},  this is equivalent to saying that $C_\mu^*(\A_\Sigma)$ is an upper semi-continuous bundle of \cstar{}algebras with fibres $C_\mu^*(\A_{\sigma^x})$. Indeed, as a direct  application of the above theorem together with \cite{BE:deformation}*{Theorem 6.16} we now get the following
 
 \begin{theorem}\label{thm-contdeform}
 Let $\A$ be a Fell bundle over $G$ and let $\Sigma=(X\times\T\into \G\onto X\times G)$ be a twist over $X\times G$
 with fibres $\sigma_x=(\T\into G_{\sigma_x}\onto G)$. Then 
 \begin{enumerate}\item If $\rtimes_\mu$ is an {\em exact} duality
 crossed-product functor
  (which always holds for $\rtimes_\max$) then $Q_x:C^*_\mu(\A_\Sigma)\onto C^*_\mu(\A_{\sigma_x})$ factors through an isomorphism 
  $C_\mu^*(\A_\Sigma)_x \cong C^*_\mu(\A_{\sigma_x})$.
\item If $G$ is an exact group, then $C_r^*(\A_\Sigma)$ is a continuous bundle of \cstar{}algebras over $X$ with fibres $C_r^*(\A_{\sigma_x})$.
\end{enumerate}
\end{theorem}

In the special case where $G$ is a {\em discrete amenable} group, the second item of the above theorem can be derived from the results in the paper \cite{Raeburn:Deformations}*{Section~6} by Iain Raeburn. Using the fact that every discrete group admits a representation group in the sense of Moore 
(see also Notation \ref{not-smooth} above), 
he used this result to show that for every discrete amenable group $G$ and Fell bundle $\A$ over $G$, there exists a kind of universal continuous bundle of \cstar{}algebras over $X=H^2(G,\T)$ (which, in this case, carries a canonical compact Hausdorff topology)
with fibres $C^*(\A_\om)$, the (unique) cross-sectional algebra of the deformed Fell bundle $\A_\om$ for $\om\in Z^2(G,\T)$ 
as considered in Proposition \ref{prop-om-Fellbundle} above. This result can now be generalized as follows

\begin{theorem}
    Suppose that $G$ is smooth in the sense of Notation \ref{not-smooth} and that $Z\into H\onto G$ is a representation group for $G$. Let $\Sigma_H$ be the twist for $\widehat{Z} \times G$ constructed in Example \ref{ex-extension}. 

    Identifying $\Twist(G)$ with $\widehat{Z}$ via the transgression map $\tg:\chi\mapsto [\sigma_\chi]$, we obtain that $C^*(\A_{\Sigma_H})$ forms an upper semi-continuous bundle of \cstar{}algebras over $\Twist(G)$, with fibres isomorphic to $C^*(\A_{\sigma})$ for $[\sigma] \in \Twist(G)$. 

    Furthermore, if $G$ is exact, then $C_r^*(\A_{\Sigma_H})$ is a continuous bundle of \cstar{}algebras over $\Twist(G)$, with fibres isomorphic to $C_r^*(\A_{\sigma})$ for $[\sigma] \in \Twist(G)$. 
\end{theorem}

\section{K-theory}

In this final section, we derive some consequences of our results in \cite{BE:deformation} concerning the $K$-theory of deformed \cstar{}algebras, applying them in the context of Fell bundles. Since we will use results from $KK$-theory, we need to restrict our considerations in this section to \emph{correspondence} crossed-product functors $\rtimes_\mu$ and assume from now on that $G$ is a {\em second-countable} locally compact group and that $\A$ is a separable Fell bundle over $G$.

We say that two twists $\sigma_0,\sigma_1$ of $G$ are \emph{homotopic}, if there exists a twist $\Sigma=([0,1]\times \T\into \G\onto [0,1]\times G)$ such that the fibres of $\Sigma$ at $0$ and $1$ are $\sigma_0$ resp. $\sigma_1$.

\begin{theorem}\label{thm-K}
Let  $G$ and $\A$ be as above and assume that $G$ satisfies the Baum-Connes conjecture with coefficients (e.g., amenable groups or, more generally,  groups with the Haagerup property). Let $\sigma_0$ and $\sigma_1$ be two homotopic twists for $G$. Then
$$K_*(C^*_r(\A_{\sigma_0}))\cong K_*(C^*_r(\A_{\sigma_1})).$$
If, in addition, $G$ is $K$-amenable, then 
$$K_*(C^*_\mu(\A_{\sigma_0}))\cong K_*(C^*_r(\A_{\sigma_0}))\cong K_*(C^*_r(\A_{\sigma_1}))\cong K_*(C^*_\mu(\A_{\sigma_1}))$$ for any correspondence crossed-product functor $\rtimes_\mu$.
\end{theorem}

The theorem is a direct consequence of \cite{BE:deformation}*{Corollary 7.6} together with our identification $C_\mu^*(A_\sigma)\cong A^\sigma_\mu$ for any twist $\sigma$ for $G$ and the cosystem $(A,\delta)=(C^*(\A), \delta_{\A})$. Indeed, if $G$ is $K$-amenable, we even have 
that $C^*_\mu(\A_{\sigma})$ is $KK$-equivalent to $C_r^*(\A_\sigma)$ for every correspondence crossed-product functor $\rtimes_\mu$ and for every twist $\sigma$ for $G$. Moreover, if $G$ satisfies the {\em strong Baum-Connes conjecture} (see \cite{BE:deformation}*{Section  7} for the notation), then 
all $K$-theory isomorphisms in the above theorem come from $KK$-equivalences.

In order to compare Theorem \ref{thm-K} with previous results on cocycle deformation: if $\om_0,\om_1\in Z^2(G,\T)$ are homotopic $2$-cocycles in the sense that 
there is a $2$-cocycle $\Omega\in Z^2(G,\cont([0,1],\T))$ whose point evaluations at $0,1\in [0,1]$ give $\Omega_0=\omega_0$ and $\Omega_1=\omega_1$, then it follows from Example \ref{eq-cocycles} that the corresponding twists $\sigma_{\om_1}$ and $\sigma_{\om_2}$ are homotopic as well. We therefore obtain analogous results for the Fell bundles  $\A_{\om_i}:=\A_{\sigma_{\om_i}}$, $i=1,2$.

\vskip 0,5pc

\end{document}